\documentclass[pdflatex,sn-mathphys]{sn-jnl}

\jyear{2021}%

\theoremstyle{thmstyleone}%
\newtheorem{theorem}{Theorem}

\newtheorem{proposition}[theorem]{Proposition}%

\theoremstyle{thmstyletwo}%
\newtheorem{example}{Example}%
\newtheorem{remark}{Remark}%

\theoremstyle{thmstylethree}%
\newtheorem{definition}{Definition}%

\DeclareMathOperator*{\Limsup}{Lim\; sup}
\def\cl{\mbox{\rm cl}\,}
\def\ir{\mbox{\rm int}\,}
\def\bd{\mbox{\rm bd}\,}
\def\tbd{\mbox{\scriptsize {\rm bd}}}

\def\epi{\mbox{\rm epi}\,}

\def\tgph{\mbox{\scriptsize {\rm gph}}\,}
\def\gph{\mbox{\rm gph}\,}
\def\dom{\mbox{\rm dom}\,}
\def\cone{\mbox{\rm cone}\,}

\def \N{\mathbb{N}}
\def \R{\mathbb{R}}
\def \B{\mathbb{B}}

\begin{document}

\title[Formulas for calculating of generalized differentials wrt a set]{Formulas for calculating of generalized differentials with respect to a set and their applications}


\author[1,2]{\fnm{Vo} \sur{Duc Thinh}}\email{vdthinh@dthu.edu.vn}

\author*[3,1]{\fnm{Xiaolong} \sur{Qin}}\email{qxlxajh@163.com}

\author[4]{\fnm{Jen-Chih} \sur{Yao}}\email{yaojc@math.nsysu.edu.tw}

\affil[1]{\orgdiv{School of Mathematical Sciences}, \orgname{Zhejiang Normal University}, \orgaddress{
\city{Jinhua}, \postcode{321004},
\state{Zhejiang}, \country{China}}}

\affil[2]{\orgname{Dong Thap University}, \orgaddress{
\city{Cao Lanh}, \postcode{81000},
\state{Dong Thap}, \country{Vietnam}}}

\affil[3]{\orgdiv{Center for Advanced Information Technology}, \orgname{Kyung Hee University}, \orgaddress{
\city{Seoul},  \country{Korea}}}

\affil[4]{\orgdiv{Research Center for Interneural Computing}, \orgname{China Medical University Hospital}, \orgaddress{
\city{Taichung}, \postcode{40402},
 \country{Taiwan}}}


\abstract{This paper provides formulas for calculating of Fr\'{e}chet and limiting normal cones with respect to a set of sets and the limiting coderivative with respect to a set of set-valued mappings. These calculations are obtained under some qualification constraints and are expressed in the similar forms of these ones of Fr\'{e}chet and limiting normal cones and the limiting coderivative. By using these obtained formulas, we state explicit necessary optimality conditions with respect to a set for optimization problems with equilibrium constraints under certain qualification conditions. Some illustrated examples to obtained results are also established.}




\maketitle

\section{Introduction and Preliminaries}
The construction of fundamental calculation or estimation methods for generalized differentials is one of the main themes of variational analysis. These calculation formulas are crucial for achieving the existence, necessary and sufficient conditions, and stability criteria for the solutions to optimization problems and optimum control problems \cite{AF91,BM07, BGO2017,Bonnans2000,ChuongKim-MOR-2016,CLSW98,DR14,G13,GO16,GM11, GM13,Mor93, Mor94,Mor06,PR98,Roc98,TC18,TC19,Yen95}.  Ginchev and   Morudkhovich recently introduced, in \cite{GM11, GM13}, several types of generalized differentials, also known as generalized differentials with respect to a set. They are known as directionally generalized differentials when the referential set is singleton. The calculation rules, estimate formulas, and certain applications of these generalized differentials were described in \cite{GM11,GM13,TC19}.  Gfrererer introduced directionally generalized differentials in \cite{G13} whose  version is sharper than the directionally generalized differentials in \cite{GM11,GM13,TC19} and is still being investigated with numerous intriguing applications (see, e.g.,  \cite{BGO2017, GO16, TC18} and the references therein). In \cite{BGO2017}, the calculation formulas for these directionally generalized differentials were provided in great detail. Recently, the authors in \cite{MLYY23} introduced new generalized differentials, known as the projectional coderivative and projectional subdifferentials. Using projectional generalized differentials, the authors presented a generalized Mordukhovich criterion of the Lipschitz-like (resp. local Lipschitz continuity) with respect to a closed convex set for set-valued functions (resp. single-valued mappings) in \cite{MLYY23} while the authors in \cite{YY22} provided necessary and sufficient conditions for Lipschitz-like property with respect to a set of the solution mappings of parameterized generalized equations due to their projectional coderivatives. These results can not be obtained via directionally limiting coderivatives. Computation rules of the projectional coderivative of set-valued mappings were established in \cite{YMLY23}. However, defining projectional generalized differentials through the projection of limiting normal cones onto tangent cones of a set makes some calculation formulas of the projectional coderivative not represented in the same form as those of the limiting coderivatives (see \cite{YMLY23}). Moreover, projectional generalized differentials may be at a {\it disadvantage} when studying optimality conditions. Very recently, the authors in \cite{TQY01-ZJNU} were presented new types of generalized differentials with respect to a set due to normal cones with respect to a set. Generalized differentials with respect to a set in \cite{TQY01-ZJNU} includes the limiting coderivative with respect to a set, the proximal and limiting subdifferentials with respect to a set. In that work, the authors also provided a generalized Mordukhovich criterion of the Aubin property (Lipschitz-like property) with respect to a closed and convex set of set-valued mappings due to the limiting coderivative with respect to a set of those set-valued mappings. Besides, the local Lipschitz continuity of single-valued mappings and necessary optimality conditions for optimization problems with geometric constraints were also stated. In particular, also in \cite{TQY01-ZJNU}, examples, which  demonstrate  that the subdifferentials with respect to a set are sharper than the limiting subdifferential and the projectional subdifferential in the study of optimality conditions, were established.

One of main goals of this work is to provide necessary optimality conditions for the following optimization problems with equilibrium constraints and additional geometric constraints:
\begin{align}\label{mpecs}
    \min \quad & f(x)\tag{MPEC}\\
    \text{such that } & 0\in G(x) \text{ \rm and } x\in \mathcal C_1\cap\mathcal C_2,\nonumber
\end{align}
where $f(x):\R^n\to\bar \R:=\R\cup\{\infty\}$, $G:\R^n\rightrightarrows\R^m$, and $\mathcal C_i\subset \R^n$ for $i=1,2$ are nonempty, closed, and convex sets.

It is important to know that the optimality conditions for   \ref{mpecs} problem in Banach spaces have been sharply established in \cite{G13} due to the directionally generalized differentials in the case that  $\mathcal C_1=\mathcal C_2=\mathbb R^n$ while the set-valued optimization version of \eqref{mpecs} has been studied in detail in \cite{BM07}
via generalized differentials, and in \cite{TC18} due to directionally generalized differentials even without the convexity of $\Omega:=\mathcal C_1\cap\mathcal C_2$.
However, optimality conditions in those articles were stated in implicit forms. For example, in \cite{G13}, under some qualification conditions, the necessary optimality conditions of \eqref{mpecs} are expressed as
$$
0\in \partial f(\bar x;u)+D^*G((\bar x,0);(u,0))(z^*),
$$
where $\bar x$ is a local optimal solution to \eqref{mpecs}, $\partial f(\cdot; u)$ and $D^*(\cdot; (u,0))$ are generalized differentials in a critical direction $u$, and $z^*\in \mathbb R^m$ is not {\it explicitly known}.

In our work, thanks to generalized differentials with respect to a set in \cite{TQY01-ZJNU}, we  provide an {\it explicit form} of optimality conditions for \ref{mpecs} problem. In particular, in cases where $G$ satisfies the Aubin property with respect to $\mathcal C_2$, the necessary optimality condition for \ref{mpecs} problem will be represented in a {\it surprisingly simple form}. We must first establish computational rules for generalized differentials with respect to a set in order to achieve this purpose. In more detail, in this article, we present the following contents:
\begin{itemize}
    \item Recall and improve the concept of normal cones with respect to a set, the limiting coderivative and subdifferentials with respect to a set in \cite{TQY01-ZJNU}.

\item Provide computation rules of normal cones with respect to a set to sets.

\item Establish formulas for calculating of the limiting coderivative with respect to a set for set-valued mappings under weak qualification conditions.

\item Use above formulas to state explicitly necessary optimality conditions for \ref{mpecs} problem under certain qualification conditions.
\end{itemize}

Through out of this paper, we assume that all of considered spaces are finite dimension spaces with a norm $\Vert\cdot\Vert$ and the scalar product $\langle \cdot,\cdot\rangle.$ Let $\mathbb R^n, \mathbb R^m, \R^s$ be finite dimension spaces with norms $\Vert \cdot\Vert_{n}$, $\Vert\cdot\Vert_m$ and $\Vert\cdot\Vert_s$, respectively. The Cartesian product $\mathbb R^n\times \mathbb R^m$ is equipped a norm defined by $\Vert (x,y)\Vert:=\Vert \cdot\Vert_n +\Vert\cdot\Vert_m.$ We sometimes use $\Vert\cdot\Vert$ for some norm if there is no the confusion. The closed unit ball in $\R^s$ is signified by $\B$ while $\B(x,r)$ is used to denote the closed ball centered at $x$ with radius $r>0.$ We define $\R^s_+:=\left\{x=(x_1,\ldots, x_s)\in\R^s\mid x_i\ge 0\; \forall i=1,\ldots, s\right\}.$ Given a real numbers sequence $(t_k)$, we write $t_k\to 0^+$ if $t_k\to 0$ and $t_k\ge 0$ for all $k.$

Let $\Omega$ be a nonempty set in $\mathbb R^s$ and $\bar x\in\Omega$. We denote the {\it interior}, the {\it closure}, and the {\it boundary} of $\Omega$ respectively by ${\rm int}\, \Omega,$ $\cl \Omega$, and $\bd \Omega.$ We denote $\cone(\Omega):=\left\{\alpha x\mid \alpha\ge 0, x\in \Omega\right\}$ and $\mathcal R(\bar x,\Omega):=\left\{x^*\mid \exists p>0: \bar x+px^*\in \Omega\right\}.$
We use the notion $x_k\xrightarrow{\Omega}\bar x$ to say that $x_k\to \bar x$ and $x_k\in \Omega$ while $x_k\xrightarrow[x_k\ne \bar x]{\Omega}\bar x$ means that $x_k\xrightarrow{\Omega}\bar x$ and $x_k\ne \bar x$ for all $k\in \N$. Let $\R^s$ be equipped the Euclidean norm. The {\it distance function} to $\Omega$, $d_{\Omega}:\mathbb R^s\to \mathbb R,$ is defined by
$$
d_{\Omega}(x):=\inf_{u\in \Omega}\Vert u-x\Vert\; \text{ \rm for all } x\in \mathbb R^s.
$$
An element $u\in \Omega$ satisfying $d_{\Omega}(x)=\Vert u-x\Vert$ is called a (Euclidean) {\it projector} (or {\it closest point}) of $x$ onto $\Omega.$ The multifunction $\Pi(\cdot, \Omega): \mathbb R^s\rightrightarrows\mathbb R^s, x\mapsto \Pi(x,\Omega):=\left\{u\in \mathbb R^s\mid \Vert u-x \Vert=d_{\Omega}(x)\right\}$ is called the {\it set-valued mapping projection} onto $\Omega$, and the set $\Pi(x,\Omega)$ is called the {\it Euclidean projector set} of $x$ onto $\Omega.$ Note that   $\Pi(x,\Omega)$ can be empty, however, if $\Omega$ is closed, then $\Pi(x,\Omega)\ne \emptyset$ for any $x\in \mathbb R^s.$ Given $u\in \Omega,$ we define $\Pi^{-1}(u,\Omega):=\left\{x\in \mathbb R^s\mid u\in \Pi(x,\Omega) \right\}.$
The {\it proximal} and {\it Fréchet normal cones} to $\Omega$ at $\bar x\in\Omega$ are respectively  given (see \cite[Definition~1.1 and page 240]{Mor06}) by
$$
N^p(\bar x,\Omega):=\cone[\Pi^{-1}(\bar x,\Omega)-\bar x]
$$
and
$$
\hat N(\bar x,\Omega):=\left\{x^*\in \mathbb R^n\mid \limsup_{x\xrightarrow{\Omega}\bar x}\dfrac{\langle x^*, x-\bar x\rangle}{\Vert x-\bar x\Vert}\le 0\right\}.
$$
The {\it limiting normal cone} to $\Omega$ at $\bar x$ is defined by
$$
N(\bar x,\Omega):=\Limsup_{x\to\bar x}\Big(\cone[x-\Pi(x,\Omega)]\Big).
$$
\begin{remark}\label{rem1}
{\rm (i) It is known in \cite[page 240]{Mor06} that $N^p(\bar x,\Omega)\subset \hat N(\bar x,\Omega)$ and }
$$
N(\bar x,\Omega)=\Limsup_{x\to \bar x}\hat N(x,\Omega)=\Limsup_{x\to \bar x} N^p(x,\Omega).
$$

 {\rm (ii) We also have that $\hat N(\bar x,\Omega)$ is a closed convex cone while $N^{p}(\bar x ,\Omega)$ is a convex cone, but it is not closed (see \cite[page~240]{Mor06}) in general. Moreover, both $\hat N(\bar x,\Omega)$ and $N(\bar x,\Omega)$ are {\it invariant} with respect to any norm on $\R^n$ but $N^{p}(\bar x ,\Omega)$ heavily depends on the Euclidean norm, i.e., if we replace the Euclidean norm by another norm, then the proximal normal cones cannot be equal (see \cite[page~10]{Mor06}).}
\end{remark}

Let $f:\R^n\to\bar\R:=\R\cup\{\infty\}$ be an extended real-valued mapping. We respectively denote the domain and the epi-graph of $f$ by
$$
\dom f:=\left\{x\in\R^n\mid f(x)<\infty\right\} \text{ \rm and } \epi f:=\left\{(x,\alpha)\mid \alpha\ge f(x)\right\}.
$$
Given $\mathcal C\subset \R^n,$ we define $f_{\mathcal C}:\R^n\to\bar\R$ by $f_{\mathcal C}(x):=\begin{cases}f(x)&\text{ \rm if } x\in \mathcal C,\\
\infty &\text{ \rm otherwise.}\end{cases}$\\
The mapping $f$ is called {\it lower semi-continuous around} $\bar x$ if $\epi f$ is a locally closed set around $(\bar x,f(\bar x))$, that is, there exists $r_0>0$ such that, for any $r\in (0,r_0)$, $\epi f\cap \B((\bar x,f(\bar x)), r)$ is a closed set. In this case, we write $f\in \mathcal F(\bar x).$

Given a set-valued mapping $F:\mathbb R^n\rightrightarrows\mathbb R^m,$ we use
$$
\dom F:=\left\{x\in \R^n\mid F(x)\ne \emptyset\right\}\; \text{ \rm and } \gph F:=\left\{(x,y)\in\R^n\times\R^m\mid y\in F(x)\right\}
$$
to signify the {\it domain} and {\it graph} of $F.$ The {\it sequential Painlev\'{e}-Kuratowski upper/outer limit} of $F$ at $\bar x\in \dom F$ is given by
$$
\Limsup_{x\to\bar x}F(x):=\left\{y\in \R^m\mid \exists x_k\to \bar x, y_k\to y \text{ \rm with } y_k\in F(x_k) \; \text{ \rm for all } k\in \N\right\}.
$$

Let $\mathcal C\subset \R^n.$ We define the set-valued mapping $F_{\mathcal C}:\R^n\rightrightarrows\R^m$ as follows:
$$
F_{\mathcal C}(x):=\begin{cases}
F(x)& \text{ \rm if } x\in \mathcal C,\\
\emptyset & \text{ \rm otherwise.}
\end{cases}
$$
Then, we have
$$
\dom F_{\mathcal C}=\dom F\cap\mathcal C \; \text{ \rm and } \;  \gph F_{\mathcal C} = \gph F\cap(\mathcal C\times\R^m).
$$
The {\it epigraphical mapping} of $F$, denoted by $\mathcal E^F:\R^n\rightrightarrows\R^m,$ is defined by
$$
\mathcal E^F(x):=F(x)+\R^m_+\; \text{ \rm for all }x\in \R^n.
$$
If $F:\R^n\rightrightarrows\R$ is a single-valued mapping, i.e., $F(x)=\{f(x)\}$ for all $x\in \R^n$, then $\gph\mathcal E^F=\epi f:=\left\{(x,\alpha)\mid \alpha\ge f(x)\right\}.$ The set-valued mapping $F$ is called to be {\it inner semicompact} {\it with respect to} $\mathcal C$ at $\bar x$ if, for any $x_k\xrightarrow{\mathcal C}\bar x$, there exists a sequence $y_k\in F(x_k)$ that contains a convergent subsequence as $k\to \infty$ while we say that $F$ is {\it inner semicontinuous} {\it with respect to} $\mathcal C$ at $(\bar x,\bar y)$ if, for any sequence $x_{k}\xrightarrow{{\rm \tiny dom} F\cap\mathcal C}\bar x,$ there exists $y_k\in F(x_k)$ for each $k\in \N$ with $y_k\to \bar y.$ $F$ is called to be {\it closed-graph} {\it with respect to} $\mathcal C$ at $\bar x$ if $\bar y\in F(\bar x)$ whenever $x_k\xrightarrow{\mathcal C}\bar x, y_k\to \bar y$ with $y_k\in F(x_k)$.

The limiting corderivative and subdifferential of multifunctions were introduced in \cite{Roc98,Mor06,BM07} based on the use of the limiting normal cone for sets as follows.

\begin{definition} {\rm Let $F:\R^n \rightrightarrows\R^m$ and $(\bar x,\bar y)\in \gph F.$

{\rm (i)} \cite[Definition~3.32]{Mor06} The {\it limiting coderivative} (also known as the {\it normal coderivative}) of $F$ at $(\bar x,\bar y)$ is the multifunction $D^*F(\bar x,\bar y): \R^m\rightrightarrows\R^n$, which is defined by
$$
D^*F(\bar x,\bar y)(y^*):=\left\{x^*\in\R^n\mid (x^*,-y^*)\in N((\bar x,\bar y),\gph F)\right\} \; \text{ \rm for all } y^*\in \R^m.
$$

{\rm (ii)} \cite[Definition~2.1]{BM07} The {\it limiting subdifferential}  of $F$ at $(\bar x,\bar y)$ is given by
$$
\partial F(\bar x,\bar y):=\left\{x^*\in\R^n\mid (x^*,-y^*)\in N((\bar x,\bar y),\gph \mathcal E^F), y^*\in \R^m_+, \Vert y^*\Vert=1\right\}.
$$

If $F$ is the single-valued mapping $f$, then we use $D^*f(\bar x)$ and $\partial f(\bar x)$ instead of $D^*f(\bar x,f(\bar x))$ and $\partial f(\bar x,f(\bar x))$, respectively.
}
\end{definition}

\section{Normal Cones and Generalized Differentials with respect to a Set}

In this section, we present the relation to a set version of the limiting normal cone, the limiting coderivative, and the limiting subdifferential. It should be noted that the concept of the limiting normal cone with respect to a set has been introduced in \cite{TQY01-ZJNU} through the proximal normal cone with respect to a set which is presented as follows: let $\Omega, \mathcal C$ be nonempty sets of the Euclidean space $\R^s$, $\bar x\in \Omega\cap \mathcal C.$ The {\it proximal normal cone} to $\Omega$ at $\bar x$ {\it with respect to} $\mathcal C$ is defined by
\begin{align}
N^p_{\mathcal C}(\bar x,\Omega):=\big\{x^*\in \mathbb R^s\mid  \exists t>0 \text{ \rm such that }& \bar x+tx^*\in \Pi^{-1}(\bar x,\Omega\cap\mathcal C)\cap\mathcal C\big\}.
\end{align}
Observe that, in \cite{TQY01-ZJNU}, $N^p_{\mathcal C}(\bar x,\Omega)$ is defined through the Euclidean projection $\Pi$, which is heavily dependent on the Euclidean norm on $\R^s$. In the following definition, we   define the limiting normal cone with respect to a set through Fr\'{e}chet normal cone with respect to that set, which does not depend on equivalent norms.

\begin{definition}\label{normal-cone-wrt}
 {\rm Let $\Omega, \mathcal C\subset \mathbb R^s$ be nonempty sets, and $\bar x\in \Omega\cap \mathcal C.$ Suppose that $\mathcal C$ is convex and locally closed around $\bar x.$ Then

{\rm (i)} The {\it Fr\'{e}chet normal cone} to $\Omega$ at $\bar x$ {\it with respect to} $\mathcal C$ is defined by
\begin{align}\label{def-fre-normal-cone}
\hat N_{\mathcal C}(\bar x,\Omega):=\left\{x^*\in \R^s\mid \exists p>0: \bar x+px^*\in \mathcal C, \limsup\limits_{x\xrightarrow{\Omega\cap\mathcal C}\bar x}\dfrac{\langle x^*, x-\bar x\rangle}{\Vert x-\bar x\Vert}\le 0\right\}.
\end{align}

{\rm (ii)} The {\it limiting normal cone} to $\Omega$ at $\bar x$ {\it with respect to} $\mathcal C$ is defined by
 \begin{equation}\label{limiting-normal-cone}
 N_{\mathcal C}(\bar x,\Omega):=\Limsup_{x\xrightarrow{\Omega\cap\mathcal C}\bar x}\hat N_{\mathcal C}(x,\Omega).\end{equation}
 }
\end{definition}

\begin{remark}
    {\rm (i) It is easy to see from \eqref{def-fre-normal-cone} that $\hat N_{\mathcal C}(\bar x,\Omega)$ (and then $N_{\mathcal C}(\bar x,\Omega)$ also) does not depend on norms on $\R^s.$}

    {\rm (ii) We can see in Proposition~\ref{pro1} that, in Euclidean spaces, the limiting normal cone with respect to a set to a locally closed set is defined as in \eqref{limiting-normal-cone} and this one in \cite{TQY01-ZJNU} are coincident. We first recall some basic results on the proximal normal cone with respect to a set which were stated in \cite{TQY01-ZJNU}.}
\end{remark}

\begin{proposition}[\cite{TQY01-ZJNU}, Proposition~1]\label{pro0}
Let $\Omega$ be a nonempty set, and  let $\mathcal C$ be a nonempty, closed, and  convex set in the Euclidean space $\mathbb R^s.$ Suppose that $\Omega$ is locally closed around $\bar x$ for the given $\bar x\in \Omega\cap \mathcal C.$ Then the following assertions hold:

    {\rm (i)} $N^p_{\mathcal C}(\bar x,\Omega)=\left\{x^*\mid \exists p>0: \bar x+px^*\in \mathcal C, \langle x^*, x-\bar x\rangle\le \dfrac{1}{2p}\Vert x-\bar x\Vert^2 \; \forall x\in \Omega\cap\mathcal C\right\}.$

    {\rm (ii)} The proximal normal cone with respect to a set can be computed by
    \begin{equation}\label{prox-nor-cone-eq2}
    N^p_{\mathcal C}(\bar x,\Omega)=N^{p}(\bar x,\Omega\cap\mathcal C)\cap\mathcal R(\bar x,\mathcal C).\end{equation}
    Consequently,
    \begin{align}
    N^p_{\mathcal C}(\bar x,\Omega)=\Big\{x^*\mid &\exists \theta,\delta,p>0 \text{ \rm such that } \bar x+px^*\in \mathcal C\nonumber\\
    &\text{ \rm and } \langle x^*, x-\bar x\rangle\le \delta\Vert x-\bar x\Vert^2 \; \forall x\in \Omega\cap\mathcal C\cap\B(\bar x,\theta)\Big\}.
    \end{align}

    {\rm (iii)} $N^p_{\mathcal C}(\bar x,\Omega)$ is a convex cone.

    {\rm (iv)} $\Limsup\limits_{x\xrightarrow{\mathcal C\cap\Omega}\bar x}N^p_{\mathcal C}(\bar x, \Omega)= \Limsup\limits_{x\xrightarrow{\mathcal C}\bar x}\left(\cone[x-\Pi(x,\Omega\cap\mathcal C)]\right).$
    \end{proposition}

We next provide several foundational properties of the Fr\'{e}chet and limiting normal cones with respect to a set.

\begin{proposition}\label{pro1}
Let $\Omega$ be a nonempty set, and let $\mathcal C$ be a nonempty, closed, and  convex set in $\mathbb R^s.$ Then, for each $\bar x\in \Omega\cap \mathcal C,$ the following assertions hold:

    {\rm (i)} $\hat N_{\mathcal C}(\bar x,\Omega)=\hat N(\bar x,\Omega\cap\mathcal C)\cap \mathcal R(\bar x,\mathcal C).$

    In addition, if $\Omega$ is convex, then
    \begin{align}
\hat N_{\mathcal C}(\bar x,\Omega)=\Big\{x^*\in \R^s\mid  \exists p>0: \bar x+px^*\in \mathcal C, \langle x^*, x-\bar x\rangle\le 0 \; \forall x\in \Omega\cap\mathcal C\Big\}.\label{formula-prox-cone-2}
\end{align}

    {\rm (ii)} $\hat N_{\mathcal C}(\bar x,\Omega)$ is a convex cone.

    {\rm (iii)} $N_{\mathcal C}(\bar x,\Omega)$ is a closed cone.

    {\rm (iv)} If $\R^s$ is an Euclidean space and $\Omega$ is locally closed around $\bar x$ then $$N_{\mathcal C}(\bar x, \Omega)= \Limsup\limits_{x\xrightarrow{\mathcal C\cap\Omega}\bar x}N^p_{\mathcal C}(\bar x, \Omega)= \Limsup\limits_{x\xrightarrow{\mathcal C}\bar x}\left(\cone[x-\Pi(x,\Omega\cap\mathcal C)]\right).$$
    \end{proposition}

\begin{proof} {\rm (i)} It directly implies from the definition and the property of Fr\'{e}chet normal cones to a convex set.

{\rm (ii)} It is easy to see that $\hat N_{\mathcal C}(\bar x,\Omega)$ is a cone. We need to prove the convexity of $\hat N_{\mathcal C}(\bar x,\Omega)$. Taking $x_1^*,x_2^*\in \hat N_{\mathcal C}(\bar x,\Omega)$ and $\lambda\in (0,1),$ we need to obtain that $x^*:=\lambda x_1^*+(1-\lambda)x_2^*\in \hat N_{\mathcal C}(\bar x,\Omega).$ Indeed, there exist, by (i), $p_1,p_2>0$ such that the following inequalities hold, for all $x\in \mathcal C\cap \Omega,$
\begin{equation*}
 \dfrac{\langle x_i^*,x-\bar x\rangle}{\Vert x-\bar x\Vert}\le 0
\end{equation*}
for $i=1,2.$ These respectively imply that
\begin{equation}\label{pro1-eq1}
\dfrac{\langle \lambda x_1^*,x-\bar x\rangle}{\Vert x-\bar x\Vert}\le 0
\end{equation}
and
\begin{equation}\label{pro1-eq2}
\dfrac{\langle (1-\lambda) x_2^*,x-\bar x\rangle}{\Vert x-\bar x\Vert}\le 0.
\end{equation}
Combining \eqref{pro1-eq1} and \eqref{pro1-eq2}, we obtain
\begin{equation*}
\dfrac{\langle x^*,x-\bar x\rangle}{\Vert x-\bar x\Vert}\le 0.
\end{equation*}
Moreover, since $\mathcal R(\bar x,\mathcal C)$ is convex, and $x^*\in \mathcal R(\bar x, \mathcal C)$, then $x^*\in\hat N_{\mathcal C}(\bar x,\Omega).$ Hence, $\hat N_{\mathcal C}(\bar x,\Omega)$ is convex.

{\rm (iii)} It is clear to see that $N_{\mathcal C}(\bar x,\Omega)$ is a cone because of the conical property of $\hat N_{\mathcal C}(\bar x,\Omega).$ Now, we prove that $N_{\mathcal C}(\bar x,\Omega)$ is closed. Taking a sequence $(x_k^*)\subset N_{\mathcal C}(\bar x,\Omega)$ satisfying $x_k^*\to x^*,$ we next justify that $x^*\in N_{\mathcal C}(\bar x,\Omega)$. Indeed, for each $k\in \N,$ there exist sequences $x_{k_m}\xrightarrow{\Omega\cap\mathcal C} x_k, x_{k_m}^*\to x_k^*$ such that $x_{k_m}^*\in \hat N_{\mathcal C}(x_{k_m},\Omega).$ For each $k\in \N$, we pick $\tilde x_k:=x_{k_k}$ and $\tilde x_k^*:=x_{k_k}^*.$ Then $\tilde x_{k}\xrightarrow{\Omega\cap\mathcal C} \bar x$, $\tilde x_k^*\to x^*$ and $\tilde x_k^*\in \hat N_{\mathcal C}(\tilde x_{k},\Omega).$ This gives us that $x^*\in N_{\mathcal C}(\bar x,\Omega)$ and hence $N_{\mathcal C}(\bar x,\Omega)$ is closed.

{\rm (iv)} By Proposition~\ref{pro0}~(iv), it is sufficiently to prove
$$
N_{\mathcal C}(\bar x,\Omega)=\Limsup_{x\xrightarrow{\mathcal C\cap\Omega}\bar x}N^p_{\mathcal C}(x,\Omega).
$$
From (i) and Proposition~\ref{pro0}~(ii), we obtain $N^p_{\mathcal C}(x,\Omega)\subset \hat N_{\mathcal C}(x,\Omega)$ for all $x\in\mathcal C\cap\Omega$. Therefore,
$$
\Limsup_{x\xrightarrow{\mathcal C\cap\Omega}\bar x}N^p_{\mathcal C}(x,\Omega)\subset \Limsup_{x\xrightarrow{\mathcal C\cap\Omega}\bar x}\hat N_{\mathcal C}(x,\Omega)=N_{\mathcal C}(\bar x,\Omega).
$$
It remains to demonstrate the opposite inclusion. Let $x^*\in N_{\mathcal C}(\bar x, \Omega).$ There are, by the definition, sequences $x_k\xrightarrow{\mathcal C\cap \Omega} \bar x$ and $x_k^*\to x^*$ such that $x_k^*\in \hat N_{\mathcal C}(x_k,\Omega)$ for all $k\in \mathbb N.$ For each $k\in \N,$ there exists $p_k>0$ such that $x_k+p_kx_k^*\in\mathcal C$. Picking $0<\alpha_k<p_k$ with $\alpha_k\to 0$, we obtain
$$
x_k+\alpha_k x_k^*=(1-\frac{\alpha_k}{p_k})x_k+\frac{\alpha_k}{p_k}(x_k+p_k x_k^*)\in \mathcal C
$$
due to the convexity of $\mathcal C$. We take $w_k\in \Pi(x_k+\alpha_kx_k^*,\Omega\cap\mathcal C).$ Then one has the following assertion
$$
\Vert w_k-x_k\Vert^2-2\alpha_k\langle x_k^*, w_k-x_k\rangle+\alpha_k^2\Vert x_k^*\Vert^2=\Vert w_k-x_k-\alpha_k x_k^*\Vert^2\le \alpha_k^2\Vert x_k^*\Vert^2
$$
which is equivalent to
\begin{equation}\label{pro1-eq00}
    \Vert w_k-x_k\Vert^2\le 2\alpha_k\langle x_k^*, w_k-x_k\rangle.
\end{equation}
Take an arbitrarily $\epsilon_k\to 0^+.$ Since $x_k^*\in \hat N_{\mathcal C}(x_k,\Omega\cap\mathcal C),$ one sees that there exists, by the definition, $\delta_k>0$ such that
\begin{equation}\label{pro1-eq01}
    \langle x_k^*, x-x_k\rangle\le \epsilon_k\Vert x-x_k\Vert\; \forall x\in \Omega\cap\mathcal C\cap\B(x_k,\delta_k).
\end{equation}
Without loss of generality, we can assume that $\alpha_k\Vert x_k^*\Vert<\frac{1}{2}\delta_k$. It implies from $w_k\in \Pi(x_k+\alpha_kx_k^*,\Omega\cap\mathcal C)$ that $w_k\in \Omega\cap\mathcal C$ and
$$
\Vert w_k-x_k-\alpha_kx_k^*\Vert\le\Vert x_k-x_k-\alpha_kx_k^*\Vert = \alpha_k\Vert x_k^*\Vert.$$ This follows that $$\Vert w_k-x_k\Vert\le 2\alpha_k\Vert x_k^*\Vert<\delta_k
$$
and so $w_k\in \B(x_k,\delta_k).$ Taking \eqref{pro1-eq01} into account, we arrive at
$$
\langle x_k^*,w_k-x_k\rangle\le \epsilon_k \Vert w_k-x_k\Vert.
$$
Combining this and \eqref{pro1-eq00}, we have $\Vert w_k-x_k\Vert\le 2\epsilon_k\alpha_k$ for all $k\in\N.$ For each $k\in \N,$ we set $w_k^*:=x_k^*+\frac{1}{\alpha_k}(x_k-w_k).$ Then $w_k+\alpha_kw_k^*=x_k+\alpha_k^*\in\mathcal C$ and
$$
\Vert w_k^*-x^*\Vert \le \Vert x_k^*-x^*\Vert+\Vert w_k^*-x_k^*\Vert = \Vert x_k^*-x^*\Vert+\frac{1}{\alpha_k}\Vert x_k-w_k\Vert\le \Vert x_k^*-x^*\Vert+2\epsilon_k.
$$
Hence $w_k^*\to x^*$ as $k\to \infty.$ On the other hand, for any $x\in \Omega\cap\mathcal C,$ we have
\begin{align}
    0&\le \Vert x_k+\alpha_k x_k^*-x\Vert^2-\Vert x_k+\alpha_k x_k^*-w_k\Vert^2\\
    &=2\langle x_k+\alpha_k x_k^*-w_k, w_k-x\rangle+\Vert w_k-x\Vert^2\\
    &=2\alpha_k\langle w_k^*, w_k-x \rangle+\Vert w_k-x\Vert^2,
\end{align}
which implies that
$$
\langle w_k^*, x-w_k \rangle\le \dfrac{1}{2\alpha_k}\Vert w_k-x\Vert^2\;\forall x\in\mathcal C\cap\Omega.$$ It results that $w_k^*\in N^p_{\mathcal C}(w_k,\Omega)$ due to $w_k+\alpha_k w_k^*\in\mathcal C.$ It  follows that $x^*\in \Limsup\limits_{x\xrightarrow{\Omega\cap\mathcal C}\bar x}N^p_{\mathcal C}(x,\Omega\cap\mathcal C).$ Therefore,
$$
N_{\mathcal C}(\bar x,\Omega)\subset \Limsup\limits_{x\xrightarrow{\Omega\cap\mathcal C}\bar x}N^p_{\mathcal C}(x,\Omega\cap\mathcal C).
$$
Hence, the proof is completed.
\end{proof}

We now recall the concept and some basic properties of variational analysis with respect to a set.

\begin{definition}[\cite{TQY01-ZJNU}]\label{coderivative-wrt}
 {\rm   Give a nonempty closed set $\mathcal C\subset \mathbb R^n$ and $\bar x\in \mathcal C$, suppose that  $F:\mathbb R^n\rightrightarrows \R^m$ has a locally closed graph and $(\bar x,\bar y)\in \gph F.$ The {\it limiting coderivative with respect to} $\mathcal C$ of $F$ at $(\bar x,\bar y)$ is a multifunction $D_{\mathcal C}^*F(\bar x,\bar y):\mathbb R^m\rightrightarrows \mathbb R^n$ defined by
$$
D^*_{\mathcal C}F(\bar x,\bar y)(y^*) = \left\{ x^*\in\mathbb R^n\mid (x^*,-y^*)\in N_{\mathcal C\times \mathbb R^m}((\bar x,\bar y),\gph F_{\mathcal C})\right\}\; \forall y^*\in \mathbb R^m.
$$
In the case of $\mathcal C=\dom F,$ we use $\bar D^*F(\bar x,\bar y)$   instead of $D^*_{\mathcal C}F(\bar x,\bar y)$ and call the {\it relative limiting coderivative} of $F$ at $(\bar x,\bar y).$  }
    \end{definition}

\begin{definition}[\cite{Roc98}, Definition~9.36]
{\rm Let $\mathcal C$ be a nonempty and  closed subset of $\mathbb R^n$, and let $F:\mathbb R^n\rightrightarrows \R^m$. Then $F$ is said  to have the {\it Aubin property with respect to} $\mathcal C$ around $(\bar x,\bar y)\in \gph F$ if $\gph F$ is locally closed around $(\bar x,\bar y)$ and there exist $\kappa\ge 0$ and neighborhoods $V$ of $\bar y$ and $U$ of $\bar x$ such that
\begin{equation}\label{Aubin-pro}
    F(u)\cap V\subset F(x) +\kappa\Vert u-x\Vert \mathbb B\quad  \forall x,u\in \mathcal C\cap U.
\end{equation}
In the case that \eqref{Aubin-pro} holds for $\mathcal C=\dom F$, we say that $F$ has the {\it relative Aubin property} around $(\bar x,\bar y).$
}
\end{definition}


\begin{definition}[\cite{Roc98}, Definition~9.1~(b)] \label{lips-vec}
{\rm Let $\mathcal C$ be a nonempty and closed subset of $\mathbb R^n.$ Let $f:\mathbb R^n\to\bar\R$ and $\bar x\in \dom f.$ Then $f$ is said  to be {\it locally Lipschitz continuous with respect to} $\mathcal C$ around $\bar x$ if
\begin{equation}\label{lips-con}
{\rm lip}_{\mathcal C}f(\bar x):=\limsup_{x,u\xrightarrow[x\ne u]{\mathcal C}\bar x}\dfrac{\vert f(x)-f(u)\vert}{\Vert x-u\Vert}< \infty.
\end{equation}
Here ${\rm lip}_{\mathcal C}f(\bar x)$ is called the {\it exact Lipschitzian constant with respect to}  $\mathcal C$ of $f$ at $\bar x.$ If \eqref{lips-con} holds with $\mathcal C=\dom f$, then $f$ is said  to be {\it relatively locally Lipschitz continuous} around $\bar x.$
}
\end{definition}

\begin{remark}  It is known in \cite{TQY01-ZJNU} that the relatively local Lipschitz continuity of a scalar mapping is different to its local Lipschitz continuity. Similarly, the relatively Aubin property of a multifunction is also different to its Aubin property.
\end{remark}

It is easy to see that the Aubin property with respect to a set of a set-valued mapping does not depend on equivalent norms of a finite dimension space. So, from Proposition~\ref{pro1} and  \cite[Corollary~4]{TQY01-ZJNU}, we have the following theorem.

\begin{theorem}[A version with respect to a set of the Mordukhovich criterion] \label{thm1}
Let $\mathcal C$ be a closed and convex subset of $\mathbb R^n$. Let $F:\mathbb R^n\rightrightarrows\mathbb R^m$, $\bar x\in \mathcal C$, and $\bar y\in F(\bar x).$ Assume that $F$ has a locally closed graph around $(\bar x,\bar y)$. Then $F$ has the Aubin property with respect to $\mathcal C$ around $(\bar x,\bar y)$ if and only if
\begin{equation}\label{thm1-eqa}
D^*_{\mathcal C}F(\bar x,\bar y)(0)=\{0\}.
\end{equation}
\end{theorem}

We now introduce the concept of subdifferentials with respect to a set of single-valued mappings.

\begin{definition}
{\rm Consider the extended real-valued function $f:\mathbb R^n\to \bar\R.$ Let $\mathcal C$ be a closed and convex subset of $\mathbb R^n$ and $\bar x\in\mathcal C.$

{\rm (i)} The set
$\hat\partial_{\mathcal C}f(\bar x):=\left\{x^*\in \mathbb R^n\mid (x^*,-1)\in \hat N_{\mathcal C\times \R}((\bar x,f(\bar x)),\epi f)\right\}$ is called the {\it Fr\'{e}chet subdifferential with resepct to} $\mathcal C$ of $f$ at $\bar x.$

{\rm (ii)} A vector $x^*\in\R^n$ satisfying $(x^*,-1)\in N_{\mathcal C\times \R}((\bar x,f(\bar x)),\epi f)$ is called a {\it limiting subgradient vector with respect to} $\mathcal C$ of $f$ at $\bar x.$  The set of all of limiting subgradient vectors with respect to $\mathcal C$ of $f$ at $\bar x$ is called the {\it limiting subdifferential with resepct to} $\mathcal C$ of $f$ at $\bar x.$ Thus we have
$$
\partial_{\mathcal C}f(\bar x):=\left\{x^*\in \mathbb R^n\mid (x^*,-1)\in N_{\mathcal C\times \R}((\bar x,f(\bar x)),\epi f_{\mathcal C})\right\}.
$$

{\rm (iii)} The {\it horizon subdifferential with resepct to} $\mathcal C$ of $f$ at $\bar x,$ denoted $\partial^{\infty}_{\mathcal C}f(\bar x)$, is defined by
$$
\partial^{\infty}_{\mathcal C}f(\bar x):=\left\{x^*\in \mathbb R^n\mid (x^*,0)\in N_{\mathcal C\times\R}((\bar x,f(\bar x)),\epi f_{\mathcal C})\right\}.
$$

If $\bar x\notin \dom f\cap\mathcal C,$ we put $\hat\partial_{\mathcal C}f(\bar x):=\partial_{\mathcal C}f(\bar x):=\partial^{\infty}_{\mathcal C}f(\bar x):=\emptyset.$ In the case of $\mathcal C=\dom f,$ we write $\tilde\partial f(\bar x),\bar\partial f(\bar x), \bar\partial^{\infty}f(\bar x)$ instead of $\hat\partial_{\mathcal C}f(\bar x), \partial_{\mathcal C}f(\bar x),\partial^{\infty}_{\mathcal C}f(\bar x)$ and say to be the {\it relative Fr\'{e}chet, limiting, horizon subdifferentials}, respectively.}
\end{definition}

The following proposition directly implies from the definitions.

\begin{proposition} \label{pro4} Let  $f:\mathbb R^n\to \bar\R$ and $\mathcal C$ be a closed and convex subset of $\mathbb R^n$. Let $\bar x\in \mathcal C.$ Then  $\partial_{\mathcal C}f(\bar x) = D^*_{\mathcal C}\mathcal E^{f}(\bar x)(1)$ and $\partial^{\infty}_{\mathcal C}f(\bar x) = D^*_{\mathcal C}\mathcal E^{f}(\bar x)(0)$.
\end{proposition}

Similar to Aubin property with respect to a set of set-valued mappings, the locally Lipschitz continuous property with respect to a set of a single-valued mapping does not depend on equivalent norms of a finite dimension space. Using \cite[Theorem~7]{TQY01-ZJNU}, Proposition~\ref{pro1} and the definition of subdifferentials with respect to a set, we obtain the following result.

\begin{theorem}\label{thm2}
Let $f:\mathbb R^n\to \bar \R$, and let $\mathcal C$ be nonempty closed convex. Let $\bar x\in \mathcal C$ and $f\in \mathcal F(\bar x).$ Then $f$ is locally Lipschitz continuous with respect to $\mathcal C$ around $\bar x$ if and only if $\partial^{\infty}_{\mathcal C}f(\bar x)=\{0\}.$
\end{theorem}

Using the limiting subdifferential with respect to a set, we provide optimality conditions for the following optimization problem:
\begin{equation}\label{optimprob1}
\min f(x) \text{ \rm such that } x\in \mathcal C,
\end{equation}
where $f:\mathbb R^n\to \bar\R$ and $\mathcal C$ is a nonempty closed convex set.

An element $\bar x\in \R^n$ is called a {\it local solution} to \eqref{optimprob1} if there exists $r>0$ such that
$$
f(\bar x)\le f(x)\; \text{ \rm for all } x\in \mathcal C\cap\B(\bar x,r).
$$
In the case of $\mathcal C=\dom f,$ we say that $\bar x$ is a {\it local minimizer} of $f$.

The following theorem provides necessary conditions for a global solution of the problem \eqref{optimprob1} due to the Fr\'{e}chet and limiting subdifferentials with respect to $\mathcal C$ of $f$.

\begin{theorem}\label{tq1-thm2}
Let $f:\mathbb R^n\to \bar\R$, and let  $\mathcal C$ be a closed convex set. Let $\bar x\in\dom f\cap\mathcal C$ and $f_{\mathcal C}\in \mathcal F(\bar x).$  If $\bar x$ is a local solution to \eqref{optimprob1}, then
$$
0\in \hat\partial_{\mathcal C}f(\bar x)\subset\partial_{\mathcal C}f(\bar x).
$$
Consequently, if $\dom f$ is a closed convex set, then $0\in \bar\partial f(\bar x)$ whenever $\bar x$ is a local minimizer of $f.$
\end{theorem}

\begin{proof}
Let $\bar x$ be a local solution to \eqref{optim-prob1}. It is sufficient to prove that $0\in \hat\partial_{\mathcal C}f(\bar x).$ Indeed, there exists $r>0$ such that
$$
f(\bar x)\le f(x)\; \text{ \rm for all } x\in\mathcal C\cap\B(\bar x,r).
$$
It follows that the following inequality holds
$$
\langle 0, x-\bar x\rangle -(r-f(\bar x))\le\epsilon(\Vert x-\bar x\Vert + \vert r-f(\bar x)\vert)
$$
 for all $(x,r)\in \epi f\cap(\mathcal C\times\R)\cap \B((\bar x,f(\bar x)),r)$ and $\epsilon>0$, and so
$$
(0,-1)\in \hat N((\bar x,f(\bar x)),\epi f\cap(\mathcal C\times\R))= \hat N((\bar x,f(\bar x)),\epi f_{\mathcal C}).
$$
Moreover, we observe that $(\bar x,f(\bar x))+p (0,-1)\in\mathcal C\times\R$ for any $p>0$. This  implies that $(0,-1)\in \hat N_{\mathcal C\times\R}((\bar x,f(\bar x)),\epi f)$, which is clearly equivalent to $0\in\hat\partial_{\mathcal C} f(\bar x).$
\end{proof}

\section{Formulas for Calculating of Normal Cones with respect to a Set}

We begin this section with the product rules for Fr\'{e}chet and limiting normal cones with respect to a set. In what follows, one always assumes that all of the considered sets are nonempty and $\mathcal C, \mathcal C_i$, for $i=1,2$, are closed and convex sets while $\Omega_i$, for $i=1,2$, are assumed to be locally closed around the reference point.

\begin{theorem}\label{thm51}
Let $\Omega_1,\mathcal C_1\subset \R^n, \Omega_2,\mathcal C_2\subset \R^m$ and let $(\bar x_1,\bar x_2)\in (\Omega_1\cap\mathcal C_1)\times (\Omega_2\cap\mathcal C_2).$ Then the following formulas hold:

{\rm (i)} $\hat N_{\mathcal C_1\times\mathcal C_2}((\bar x_1,\bar x_2), \Omega_1\times\Omega_2)=\hat N_{\mathcal C_1}(\bar x_1,\Omega_1)\times \hat N_{\mathcal C_2}(\bar x_2,\Omega_2).$

{\rm (ii)} $N_{\mathcal C_1\times\mathcal C_2}((\bar x_1,\bar x_2), \Omega_1\times\Omega_2)=N_{\mathcal C_1}(\bar x_1,\Omega_1)\times N_{\mathcal C_2}(\bar x_2,\Omega_2).$
\end{theorem}

\begin{proof}
{\rm (i)} Thank to \cite[Proposition~1.2]{Mor06}, one has
\begin{align}
\hat N((\bar x_1,\bar x_2),\Omega_1\times\Omega_2\cap\mathcal C_1\times\mathcal C_2)&=\hat N((\bar x_1,\bar x_2),\Omega_1\times\mathcal C_1\cap\Omega_2\times\mathcal C_2)\nonumber\\
&=\hat N(\bar x_1,\Omega_1\cap\mathcal C_1)\times \hat N(\bar x_2,\Omega_2\cap\mathcal C_2).\label{thm51-eq1}
\end{align}
We next demonstrate
$$
\mathcal R((\bar x_1,\bar x_2), \mathcal C_1\times\mathcal C_2)=\mathcal R(\bar x_1, \mathcal C_1)\times\mathcal R(\bar x_2, \mathcal C_2).
$$
Indeed, taking $(x_1^*,x_2^*)\in \mathcal R((\bar x_1,\bar x_2), \mathcal C_1\times\mathcal C_2)$, one sees that  there exist $p>0$ such that
$$
(\bar x_1,\bar x_2)+p(x_1^*,x_2^*)\in \mathcal C_1\times\mathcal C_2,
$$
which implies that $x_1^*\in \mathcal R(\bar x_1,\mathcal C_1)$ and $x_2^*\in \mathcal R(\bar x_2,\mathcal C_2)$. Therefore,
$$
\mathcal R((\bar x_1,\bar x_2),\mathcal C_1\times\mathcal C_2)\subset\mathcal R(\bar x_1,\mathcal C_1)\times\mathcal R(\bar x_2, \mathcal C_2).
$$
To see the opposite inclusion, one takes $$(u_1^*,u_2^*)\in \mathcal R(\bar x,\mathcal C_1)\times\mathcal R(\bar x_2, \mathcal C_2).$$ There we find $\lambda_i> 0$ such that $\bar x_i+\lambda_iu_i^*\in \mathcal C_i$ for $i=1,2$. Putting $\lambda=\min\{\lambda_1,\lambda_2\},$ we obtain $\bar x_i+\lambda u_i^*\in \mathcal C_i$ due to the convexity of $\mathcal C_i$ and the fact $\bar x_i\in \mathcal C_i$ for $i=1,2.$ It follows that $(\bar x_1,\bar x_2)+\lambda(u_1^*,u_2^*)\in\mathcal C_1\times\mathcal C_2$ and then $(u_1^*,u_2^*)\in \mathcal R((\bar x_1,\bar x_2),\mathcal C_1\times\mathcal C_2).$ Thus, we derive
$$
\mathcal R((\bar x_1,\bar x_2),\mathcal C_1\times\mathcal C_2)\supset\mathcal R(\bar x_1,\mathcal C_1)\times\mathcal R(\bar x_2, \mathcal C_2).
$$
So, we obtain the following equality
$$
\mathcal R((\bar x_1,\bar x_2),\mathcal C_1\times\mathcal C_2)=\mathcal R(\bar x_1,\mathcal C_1)\times\mathcal R(\bar x_2, \mathcal C_2).
$$
Taking \eqref{thm51-eq1} into account and using the definition, we have
\begin{align*}
        &\hat N_{\mathcal C_1\times\mathcal C_2}((\bar x_1,\bar x_2), \Omega_1\times\Omega_2)\\
        &=\hat N((\bar x_1,\bar x_2), \Omega_1\times\Omega_2\cap\mathcal C_1\times\mathcal C_2)\cap \mathcal R((\bar x_1,\bar x_2),\mathcal C_1\times\mathcal C_2)\\
        &=\left(\hat N(\bar x_1,\Omega_1\cap\mathcal C_1)\times \hat N(\bar x_2,\Omega_2\cap\mathcal C_2)\right)\cap\left(\mathcal R(\bar x_1,\mathcal C_1)\times\mathcal R(\bar x_2, \mathcal C_2)\right)\\
        &=\left(\hat N(\bar x_1,\Omega_1\cap\mathcal C_1)\cap \mathcal R(\bar x_1,\mathcal C_1)\right)\times \left(\hat N(\bar x_2,\Omega_2\cap\mathcal C_2)\cap\mathcal R(\bar x_2, \mathcal C_2)\right)\\
        &=\hat N_{\mathcal C_1}(\bar x_1,\Omega_1)\times\hat N_{\mathcal C_2}(\bar x_2,\Omega_2).
\end{align*}

{\rm (ii)} Let $(x_1^*,x_2^*)\in N_{\mathcal C_1\times\mathcal C_2}((\bar x_1,\bar x_2),\Omega_1\times\Omega_2).$ There exist by the definition sequences $(u_k,v_k)\xrightarrow{\Omega_1\times\Omega_2\cap\mathcal C_1\times\mathcal C_2} (\bar x_1,\bar x_2)$ and $(u_k^*, v_k^*)\to (x_1^*,x_2^*)$ such that
$$
(u_k^*, v_k^*)\in \hat N_{\mathcal C_1\times\mathcal C_2}((u_k,v_k),\Omega_1\times\Omega_2).
$$
By using (i), we achieve
$$
(u_k^*, v_k^*)\in \hat N_{\mathcal C_1}(u_k,\Omega_1)\times \hat N_{\mathcal C_2}(v_k,\Omega_2).
$$
Since $u_k\xrightarrow{\Omega_1\cap\mathcal C_1}\bar x_1$, $u_k^*\to x^*_1$ and $u_k^*\in \hat N_{\mathcal C_1}(u_k,\Omega_1),$ we have $x_1^*\in N_{\mathcal C_1}(\bar x_1,\Omega_1).$ Similarly, we also have $x_2^*\in N_{\mathcal C_2}(\bar x_1,\Omega_2)$ and then
$$
(x_1^*,x_2^*)\in N_{\mathcal C_1}(\bar x_1,\Omega_1)\times N_{\mathcal C_2}(\bar x_2,\Omega_2).
$$
Therefore,
$$
N_{\mathcal C_1\times\mathcal C_2}((\bar x_1,\bar x_2),\Omega_1\times\Omega_2)\subset N_{\mathcal C_1}(\bar x_1,\Omega_1)\times N_{\mathcal C_2}(\bar x_2,\Omega_2).
$$
It remains to prove the opposite inclusion. Taking $(y_1^*,y_2^*)\in N_{\mathcal C_1}(\bar x_1,\Omega_1)\times N_{\mathcal C_2}(\bar x_2,\Omega_2),$ we find by the definition sequences $z_k\xrightarrow{\Omega_1\cap\mathcal C_1}\bar x_1, z_k^*\to y_1^*$ and $w_k\xrightarrow{\Omega_2\cap\mathcal C_2}\bar x_2, w_k^*\to y_2^*$ such that $z_k^*\in \hat N_{\mathcal C_1}(z_k,\Omega_1)$ and $w_k^*\in \hat N_{\mathcal C_2}(w_k,\Omega_2),$ which  implies that  $(z_k,w_k)\xrightarrow{\Omega_1\times\Omega_2\cap\mathcal C_1\times\mathcal C_2}(\bar x_1,\bar x_2)$ and $(z_k^*,w_k^*)\to (y_1^*,y_2^*)$ with
$$
(z_k^*,w_k^*)\in \hat N_{\mathcal C_1}(z_k,\Omega_1)\times\hat N_{\mathcal C_2}(w_k,\Omega_2)=\hat N_{\mathcal C_1\times\mathcal C_2}((z_k,w_k),\Omega_1\times\Omega_2).
$$
This follows by the definition that $(y_1^*,y_2^*)\in N_{\mathcal C_1\times\mathcal C_2}((\bar x_1,\bar x_2),\Omega_1\times\Omega_2)$ and so
$$
N_{\mathcal C_1\times\mathcal C_2}((\bar x_1,\bar x_2),\Omega_1\times\Omega_2)\supset N_{\mathcal C_1}(\bar x_1,\Omega_1)\times N_{\mathcal C_2}(\bar x_2,\Omega_2).
$$
Hence, we obtain the following relation
$$
N_{\mathcal C_1\times\mathcal C_2}((\bar x_1,\bar x_2),\Omega_1\times\Omega_2) = N_{\mathcal C_1}(\bar x_1,\Omega_1)\times N_{\mathcal C_2}(\bar x_2,\Omega_2).
$$
The proof of theorem is completed.
\end{proof}

\begin{theorem}\label{product-rule2}
Let $\Omega_1,\mathcal C_1\subset\R^{n+s}$ and $\Omega_2,\mathcal C_2\subset \R^m$. Define the following sets  $$\Omega:=\{(x,y,z)\in \R^n\times\R^m\times\R^s\mid (x,z)\in \Omega_1, y\in\Omega_2\}$$ and
$$
\mathcal C:=\{(x,y,z)\in \R^n\times\R^m\times\R^s\mid (x,z)\in \mathcal C_1, y\in\mathcal C_2\}
$$
and sssume that $\Omega_i$ for $i=1,2$ are locally closed around $(\bar x,\bar y,\bar z)\in \Omega\cap\mathcal C.$ Then

{\rm (i)} $\hat N_{\mathcal C}((\bar x,\bar y,\bar z),\Omega)=\{(x^*,y^*,z^*)\mid (x^*,y^*)\in \hat N_{\mathcal C_1}((\bar x,\bar z),\Omega_1), y^*\in \hat N_{\mathcal C_2}(\bar y,\Omega_2)\}.$

{\rm (ii)} $N_{\mathcal C}((\bar x,\bar y,\bar z),\Omega)=\{(x^*,y^*,z^*)\mid (x^*,y^*)\in N_{\mathcal C_1}((\bar x,\bar z),\Omega_1), y^*\in N_{\mathcal C_2}(\bar y,\Omega_2)\}.$
\end{theorem}
\begin{proof}
It is sufficiently to prove (i) because (ii) implies from (i) with the similar arguments as in the proof of Theorem~\ref{thm51}~(ii). We first observe that
\begin{equation}\label{product-rule2-eq1}
    \hat N_{\mathcal C}((\bar x,\bar y,\bar z),\Omega)=\hat N((\bar x,\bar y,\bar z),\Omega\cap\mathcal C)\cap\mathcal R((\bar x,\bar y,\bar z),\mathcal C).
\end{equation}
Since $\Omega\cap\mathcal C=\{(x,y,z)\mid (x,z)\in \Omega_1\cap\mathcal C_1, y\in \Omega_2\cap\mathcal C_2\},$ we have
\begin{equation}\label{product-rule2-eq2}
    \hat N((\bar x,\bar y,\bar z),\Omega\cap\mathcal C)=\{(x^*,y^*,z^*)\mid (x^*,z^*)\in \hat N((\bar x,\bar z),\Omega_1\cap\mathcal C_1), y^*\in \hat N(\bar y,\Omega_2\cap\mathcal C_2)\}.
\end{equation}
Moreover, we also have
\begin{align*}
        \mathcal R((\bar x,\bar y,\bar z), \mathcal C)&=\{(u,v,w)\mid \exists p> 0 : (\bar x,\bar y,\bar z)+p(u,v,w)\in\mathcal C\}\\
        &=\Big\{(u,v,w)\mid \exists p> 0 : (\bar x,\bar z) +p(u,w)\in \mathcal C_1, \bar y+pv\in \mathcal C_2\Big\}\\
        &\subset \{(u,v,w)\mid (u,w)\in \mathcal R((\bar x,\bar z),\mathcal C_1), v\in \mathcal R(\bar y,\mathcal C_2)\}.
\end{align*}

On the other hand, let
$$
(\bar u,\bar v,\bar w)\in \{(u,v,w)\mid (u,w)\in \mathcal R((\bar x,\bar z),\mathcal C_1), v\in \mathcal R(\bar y,\mathcal C_2)\}.
$$
Then there exist $p_1,p_2> 0$ such that
$$
(\bar x,\bar z)+p_1(\bar u,\bar w)\in\mathcal C_1 \text{ \rm and } \bar y+p_2\bar v\in\mathcal C_2.
$$
Putting $p:=\min\{p_1,p_2\},$ we have $(\bar x,\bar z)+p(\bar u,\bar w)\in\mathcal C_1$ and $\bar y+p\bar v\in\mathcal C_2$ due to the convexity of $\mathcal C_1,\mathcal C_2.$ Therefore,
$$
(\bar x,\bar y,\bar z)+p(\bar u,\bar v,\bar w)\in \mathcal C,
$$
which gives us that
$$
(\bar u,\bar v,\bar w)\in \mathcal R((\bar x,\bar y,\bar z),\mathcal C).
$$
Thus we obtain
$$
\{(u,v,w)\mid (u,w)\in \mathcal R((\bar x,\bar z),\mathcal C_1), v\in \mathcal R(\bar y,\mathcal C_2)\}\subset \mathcal R((\bar x,\bar y,\bar z),\mathcal C)
$$
and hence
$$
\{(u,v,w)\mid (u,w)\in \mathcal R((\bar x,\bar z),\mathcal C_1), v\in \mathcal R(\bar y,\mathcal C_2)\}=\mathcal R((\bar x,\bar y,\bar z),\mathcal C).
$$
Taking \eqref{product-rule2-eq1} and \eqref{product-rule2-eq2} into account, we achieve
\begin{align*}
    \hat N_{\mathcal C}((\bar x,\bar y,\bar z),\Omega)=\{(x^*,y^*,z^*)\mid (x^*,z^*)\in \hat N_{\mathcal C_1}((\bar x,\bar z),\Omega_1), y^*\in \hat N_{\mathcal C_2}(\bar y,\Omega_2)\}.
\end{align*}
Hence, the proof is completed.
\end{proof}

We next establish the sum rules for normal cones with respect to a set. To do this, we need to present two qualification conditions as follows.

\begin{definition}\label{slqc}
{\rm Let $\Omega_1,\Omega_2,\mathcal C_1$ and $\mathcal C_2$ be nonempty sets, and let $\bar x\in \Omega_1\cap\Omega_2\cap\mathcal C_1\cap\mathcal C_2.$

{\rm (i)} We say that $\{\Omega_1,\Omega_2\}$ satisfy the {\it limiting qualification condition with respect to} (LQC wrt) $\{\mathcal C_1,\mathcal C_2\}$  at $\bar x$ if, for any $x_{ik}\xrightarrow{\Omega_i}\bar x$ and $x_{ik}^*\in \hat N_{\mathcal C_i}(x_{ik},\Omega_i)$ with $x_{ik}^*\to x_i^*$ for $i=1,2$,
$$
x_{1k}^*+x_{2k}^*\to 0\implies x_i^*= 0 \text{ \rm for } i=1,2.
$$

{(ii)} Let $\Omega:=\Omega_1\cap\Omega_2$ and $\mathcal C:=\mathcal C_1\cap\mathcal C_2$. We say that $\{\Omega_1,\Omega_2\}$ are {\it normal-densed} in $\{\mathcal C_1,\mathcal C_2\}$ at $\bar x$ if whenever there are sequences $x_{ik}\xrightarrow{\Omega_i\cap\mathcal C_i}\bar x, x_k\xrightarrow{\Omega\cap\tbd \mathcal C}\bar x$ and $x_{ik}^*\in \hat N(x_{ik},\Omega_i\cap\mathcal C_i)$ for $i=1,2$, $x_k^*\in \mathcal R(x_k,\mathcal C)$ satisfying $x_{ik}^*\to x_i^*$ with some $x_i^*$ for $i=1,2$ and $x_k^*\to x_{1}^*+x_{2}^*,$ we find sequences $\tilde x_{ik}\xrightarrow{\Omega_i\cap\mathcal C_i}\bar x$ and $\tilde x_{ik}^*\in \hat N_{\mathcal C_i}(\tilde x_{ik}, \Omega_i)$ for $i=1,2$  such that $\tilde x_{ik}^*\to \tilde x_i^*$ with some $\tilde x_i^*$ for $i=1,2$ and
 \begin{equation}\label{slqc-eq00}
     \tilde x_{1}^*+\tilde x_{2}^*=x_1^*+x_2^*
 \end{equation} and $\max\{\Vert \tilde x_{1}^*\Vert, \Vert \tilde x_{2}^*\Vert\}>0 \text{ \rm whenever } \max\{\Vert x_{1}^*\Vert, \Vert x_{2}^*\Vert\} >0.$
 }
\end{definition}

\begin{remark}\label{rem-slqc}
{\rm (i) In the case that $\mathcal C_i$ for $i=1,2$ are neighborhoods of $\bar x$, the limiting qualification condition with respect to $\{\mathcal C_1,\mathcal C_2\}$ reduces to the limiting qualification condition, which is presented in infinite dimension spaces in \cite[Definition~3.2~(ii)]{Mor06}.}

{\rm (ii)    In Definition~\ref{slqc}~(ii), if $x_{ik}\in {\rm int}\,\mathcal C_i$ for $i=1,2$, then $\mathcal R(x_{ik},\mathcal C_i)=\R^n$. Hence, $\{\Omega_1,\Omega_2\}$  are normal-densed in $\{\mathcal C_1,\mathcal C_2\}$ at $\bar x$.}

{\rm (iii) In Section~\ref{sec-app}, we will show that the constraint qualification in Definition~\ref{slqc}~(ii)  is {\it essential} to obtain computation rules for normal cones and to provide optimality conditions for \ref{mpecs} problem. The following example illustrates the qualification conditions in Definition~\ref{slqc}.
}
\end{remark}


\begin{example}\label{exam1}
	Consider the following sets
	$$\Omega_1:=\left\{(x,y,z)\mid y\in \R, z\ge x \right\},\; \Omega_2:=\R^2_+\times \R,\;  \mathcal C_1:=\R^3, \text{ \rm and } \mathcal C:=\mathcal C_2:=\R_+\times\R^2.$$
	Then the following assertions hold:
	
	{\rm (i)} $\{\Omega_1,\Omega_2\}$ satisfy LQC wrt $\{\mathcal C,\mathcal C\}$ at $(0,0,0)$ and  $\{\Omega_1,\Omega_2\}$ are normal-densed in $\{\mathcal C,\mathcal C\}$ at $(0,0,0).$
	
	{\rm (ii)} $\{\Omega_1,\Omega_2\}$ satisfy LQC wrt $\{\mathcal C_1,\mathcal C_2\}$ at $(0,0,0)$  but $\{\Omega_1,\Omega_2\}$ are not  normal-densed in $\{\mathcal C_1,\mathcal C_2\}$ at $(0,0,0).$
	
	We first show (i). Let $(x,y,z)\in \Omega_i\cap\mathcal C$ for $i=1,2$ be sufficiently close $(0,0,0).$ By directly computation, we achieve:
	\begin{equation*}
	\hat N((x,y,z),\Omega_1\cap\mathcal C)	=\left\{(u,0,v)\left\vert
		\begin{matrix}
			(u,v)\in\R_-^2\cup\{(u,v)\mid 0\le u\le -v \} & \text{\rm if } (x,z)=0;\\
			0\le u=-v  &\text{\rm if } x=z>0;\\
			u\le 0, v=0 & \text{\rm if } x=0,z>0;\\
			(u,v)=(0,0)& \text{\rm if } z>x>0
		\end{matrix}\right.
		\right\}
	\end{equation*}
	
	\begin{equation}\label{exam1-norm1}
	\hat N_{\mathcal C}((x,y,z),\Omega_1)	=\left\{(u,0,v)\left\vert
		\begin{matrix}
			u=0, v\le 0 \text{ \rm or } 0\le u\le -v  & \text{ \rm if } (x,z)=0;\\
			0\le u=-v  &\text{ \rm if } x=z>0;\\
			(u,v)=(0,0)& \text{ \rm if } z>x\ge 0
		\end{matrix}\right.
		\right\}
	\end{equation}
	\begin{equation*}
		\hat N((x,y,z),\Omega_2\cap\mathcal C)	=
		\begin{cases}
			\R_-^2\times\{0\}&\text{ \rm if } x=y=0;\\
			\R_-\times\{0\}^2&\text{ \rm if } x=0,y>0;\\
			\{0\}\times\R_-\times\{0\}&\text{ \rm if } x>0,y=0;\\
			\{(0,0,0)\}&\text{ \rm if } x,y>0.
		\end{cases}
	\end{equation*}
	\begin{equation}\label{exam1-norm2}
		\hat N_{\mathcal C}((x,y,z),\Omega_2)	=
		\begin{cases}
			\{0\}\times\R_-\times\{0\}&\text{ \rm if }x\ge 0, y=0;\\
			\{(0,0,0)\}&\text{ \rm if } x\ge 0, y>0.
		\end{cases}
	\end{equation}
It is clearly, from \eqref{exam1-norm1} and \eqref{exam1-norm2}, that $\{\Omega_1,\Omega_2\}$ satisfy LQC wrt $\{\mathcal C,\mathcal C\}$ at $(0,0,0).$

To prove that $\{\Omega_1,\Omega_2\}$ are normal-densed in $\{\mathcal C,\mathcal C\}$ at $(0,0,0),$ we set $\Omega:=\Omega_1\cap\Omega_2.$  For any sequence $(x_k,y_k,z_k)\xrightarrow{\Omega\cap\tbd\mathcal C} (0,0,0),$ we have $x_k=0$ and thus $\mathcal R((x_k,y_k,z_k),\mathcal C) = \R_+\times\R^2.$ Take sequences $(x_{ik},y_{ik},z_{ik})\xrightarrow{\Omega_i\cap\mathcal C} (0,0,0)$, $(x^*_{ik},y^*_{ik},z^*_{ik})\in \hat N((x_{ik},y_{ik},z_{ik}),\Omega_i\cap\mathcal C)$ with $(x^*_{ik},y^*_{ik},z^*_{ik})\to (x_i^*,y_i^*,z_i^*)$ for some $(x_i^*,y_i^*,z_i^*)\in\R^3$ and $i=1,2,$ and $(x_k^*,y_k^*,z_k^*)\in \mathcal R((x_k,y_k,z_k),\mathcal C)$ satisfying $$(x^*_{k},y^*_{k},z^*_{k})\to  (x_1^*,y_1^*,z_1^*)+(x_2^*,y_2^*,z_2^*).$$ We will justify sequences $(\tilde x_{ik},\tilde y_{ik},\tilde z_{ik})\xrightarrow{\Omega_i\cap\mathcal C} (0,0,0)$, $(\tilde x^*_{ik},\tilde y^*_{ik},\tilde z^*_{ik})\in \hat N_{\mathcal C}((x_{ik},y_{ik},z_{ik}),\Omega_i)$ with $(\tilde x^*_{ik},\tilde y^*_{ik},\tilde z^*_{ik})\to (\tilde x_i^*,\tilde y_i^*,\tilde z_i^*)$ for some $(\tilde x_i^*,\tilde y_i^*,\tilde z_i^*)\in\R^3$ and $i=1,2,$ satisfying
\begin{align}
&\qquad\qquad\qquad (x_1^*,y_1^*,z_1^*)+(x_2^*,y_2^*,z_2^*)=  (\tilde x_1^*,\tilde y_1^*,\tilde z_1^*)+(\tilde x_2^*,\tilde y_2^*,\tilde z_2^*),\label{exam1eq1}\\
&\max\{\Vert(\tilde x_1^*,\tilde y_1^*,\tilde z_1^*)\Vert, \Vert(\tilde x_2^*,\tilde y_2^*,\tilde z_2^*)\Vert\}>0 \text{ \rm if }  \max\{\Vert(x_1^*,y_1^*,z_1^*)\Vert, \Vert(x_2^*,y_2^*,z_2^*)\Vert\}>0\label{exam1eq2}
\end{align}
This trivially holds if $\hat N((x_{ik},y_{ik},z_{ik}),\Omega_i\cap\mathcal C)=\hat N_{\mathcal C}((x_{ik},y_{ik},z_{ik}),\Omega_i)$ for $i=1,2$. Thus we only need to consider the following cases:
	\begin{itemize}
		\item[(1)] $(x_{2k},y_{2k})=(0,0), z_{2k}\to 0$. In this case, we have $(x^*_{2k},y^*_{2k},z^*_{2k})\in\R^2_-\times\{0\}$ which implies $x_2^*\le 0, y_2^*\le 0, z_2^*=0.$  Let us consider the following subcase.
		\item[(1.1)] $(x_{1k},z_{1k}) =(0,0)$ and $y_{1k}\to 0$.
		
		\item[(a)] If $(x_{1k}^*,y_{1k}^*,z_{1k}^*)\in \R_-\times\{0\}\times\R_-$ then we have $x_{1}^*=x_{2}^*=0.$ By taking $(\tilde x_{ik},\tilde y_{ik},\tilde z_{ik})=(x_{ik},y_{ik},z_{ik})$ and $(\tilde x_{ik}^*,\tilde y_{ik}^*,\tilde z_{ik}^*)=(0,y_{ik}^*,z_{ik}^*)\to (0,y_i^*,z_i^*)$ for $i=1,2$, we find sequences satisfying \eqref{exam1eq1} and \eqref{exam1eq2}.
		
		\item[(b)] If $0\le x_{1k}^*\le -z_{1k}^*$ then we take $(\tilde x_{ik},\tilde y_{ik},\tilde z_{ik})=(0,0,0)$ for $i=1,2$ and $(\tilde x_{2k}^*,\tilde y_{2k}^*, \tilde z_{2k}^*) = (0,y_{2k}^*,0)\to (0,y_2^*,0)$ and $(\tilde x_{1k}^*,\tilde y_{1k}^*,\tilde z_{1k}^*)=(x_{1}^*+x_{2}^*, 0,z_{1}^*).$ Let $\max\{\Vert (x_1^*,0,z_1^*)\Vert, \Vert (x_2^*,y_2^*,0)\Vert\}>0.$ We have the following two possibilities. If  $\Vert (x_1^*,0,z_1^*)\Vert>0$ then $\vert z_1^*\vert>0$ because  $0\le x_1^*\le -z_1^*.$ It implies that $\Vert (x_{1}^*+x_{2}^*, 0,z_{1}^*)\Vert>0.$ Otherwise, if $\Vert (x_2^*,y_2^*,0)\Vert>0$ then $\Vert(0,y_2^*,0)\Vert>0$  or $\vert x_2^*\vert>0$ which gives us that $\Vert (x_1^*+x_2^*,0,z_1^*)\Vert>0$ due to $x_1^*\ge 0.$ So, in these all cases, the sequences $(\tilde x_{ik},\tilde y_{ik},\tilde z_{ik})$ and $(\tilde x_{ik}^*,\tilde y_{ik}^*,\tilde z_{ik}^*)$ for $i=1,2$ satisfy \eqref{exam1eq1} and \eqref{exam1eq2}.
		\item[(1.2)]  $x_{1k}=z_{1k} >0$ and $y_{1k}\in\R$. We have $0\le x^*_{1k}=-z_{1k}^*, y_{1k}^*=0$ and so $0\le x_1^*=-z_1^*, y_1^*=0$. Take $\tilde x_{1k}=\tilde z_{1k}=0, \tilde y_{1k}=y_{1k}$, $(\tilde x_{2k},\tilde y_{2k},\tilde z_{2k}) =(x_{2k},y_{2k},z_{2k})$  and $\tilde x_{1k}^*=x_{1k}^*+x_{2k}^*$, $\tilde y_{1k}^*=0$, $\tilde z_{1k}^*=z_{1k}^*\le -\tilde x_{1k}^*$, $\tilde x_{2k}^*=0,$ $\tilde y_{2k}^*=y_{2k}^*$, $\tilde z_{2k}^*=0$.  Then $(\tilde x_{ik}^*, \tilde y_{ik}^*,\tilde z_{ik}^*)\in \hat N_{\mathcal C}((\tilde x_{ik}, \tilde y_{ik},\tilde z_{ik}),\Omega_i)$ for $i=1,2$, $(\tilde x_{1k}^*, \tilde y_{1k}^*,\tilde z_{1k}^*)\to (x_1^*+x_2^*,0,z_1^*),$ $(\tilde x_{2k}^*, \tilde y_{2k}^*,\tilde z_{2k}^*)\to (0,y_2^*,0)$ and $$(x_1^*+x_2^*,0,z_1^*)+(0,y_2^*,0)= (x_1^*,0,z_1^*)+(x_2^*,y_2^*,0).$$ Moreover, if $\max\{\Vert (x_{i}^*,y_i^*, z_i^*)\Vert\mid i=1,2\}>0$ then by the similar arguments as (1.1. (b)), we have $\max\{\Vert (\tilde x_{i}^*, \tilde y_i^*, \tilde z_i^*)\Vert\mid i=1,2\}>0.$ Thus the above sequences satisfy \eqref{exam1eq1} and \eqref{exam1eq2}.
		\item[(1.3)] $x_{1k}=0,z_{1k} >0$ and $y_{1k}\to 0$. In this case, we have $x_1^*=x_2^*=0.$ Using similar arguments as (1.1. (a)), we have sequences satsifying \eqref{exam1eq1} and \eqref{exam1eq2}.
		\item[(1.4)] $x_{1k}>0,z_{1k} >0$ and $y_{1k}\to 0$. We have $x_{1k}^*=z_{1k}^*=0$ and thus $x_1^*=y_1^*=z_1^*=0.$  It implies that $x_2^*=0.$ It is easy to see that the sequences $(\tilde x_{ik},\tilde y_{ik},\tilde z_{ik})=(x_{ik},y_{ik},z_{ik})$ and $(\tilde x_{1k}^*,\tilde y_{1k}^*,\tilde z_{1k}^*)=(0,0,0)$, $(\tilde x_{2k}^*,\tilde y_{2k}^*,\tilde z_{2k}^*)=(0,y_{2k}^*,z_{2k}^*)$ satisfy \eqref{exam1eq1} and \eqref{exam1eq2}.
		\item[(2)] $x_{2k}=0, y_{2k}>0$ is proven in a similar way to (1).
		\item[(3)] $x_{2k}>0, y_{2k}=0$. In this case, $\hat N((x_{2k},y_{2k},z_{2k}),\Omega_2\cap\mathcal C)=\{0\}\times\R_-\times\{0\}.$ We only need to consider the following subcases:
		\item[(3.1)] $x_{1k}=z_{1k}=0$, $y_{1k}\to 0$ and $(x_{1k}^*,z_{1k}^*)\in\R_-^2.$ In this case, we have $x_1^*\le 0$ and $x_2^*\le 0.$ Since $\R_+\ni x_k^*\to x_1^*+x_2^*,$ we obtain $x_{1}^*=x_{2}^*=0$ and thus we can use the similar arguments as in (1.1. (a)).
		\item[(3.2)] $x_{1k}=0,z_{1k}>0$, $y_{1k}\to 0$ and $x_{1k}^*<0,z_{1k}^*=0.$ It also implies that $x_1^*=x_2^*=0$ and thus we use the arguments of (1.1. (a)) again.
		\item[(4)] $x_{2k},y_{2k}>0, z_{2k}\to 0.$ In this case, $x_{2k}^*=y_{2k}^*=z_{2k}^*=0$ and so $x_2^*=y_2^*=z_2^*=0.$ 
		\item[(4.1)] $x_{1k}=z_{1k}=0$, $y_{1k}\to 0$ and $(x_{1k}^*,z_{1k}^*)\in\R_-^2.$ In this case, we have $x_1^*\le 0.$  Since $\R_+\ni x_k^*\to x_1^*+x_2^*,$ we obtain $x_{1}^*=0$ also. Taking $(\tilde x_{ik},\tilde y_{ik},\tilde z_{ik}) = (x_{ik}, y_{ik}, z_{ik})$ for $i=1,2$ and $\tilde x_{1k}^*=\tilde x_{2k}^*=0, \tilde y_{1k}^*=y_{1k}^*, \tilde y_{2k}^*=0, \tilde z_{1k}^*=z_{1k}^*, \tilde z_{2k}^*=0$ and using the similar arguments as in (1.1. (a)), we derive that these sequences satisfy \eqref{exam1eq1} and \eqref{exam1eq2}.
		\item[(4.2)] $x_{1k}=0,z_{1k}>0$, $y_{1k}\to 0$ and $x_{1k}^*<0,z_{1k}^*=0.$ Using the similar arguments as in (4.1), we find sequences satisfying \eqref{exam1eq1} and \eqref{exam1eq2}.
			\end{itemize}
		From the cases (1)-(4), we alway find  the sequences $(\tilde x_{ik},\tilde y_{ik},\tilde z_{ik})$ and $(\tilde x_{ik}^*,\tilde y_{ik}^*,\tilde z_{ik}^*)$ for $i=1,2$ satisfy \eqref{exam1eq1} and \eqref{exam1eq2} and hence (i) holds.
		
		\medskip
		
		We remain to prove (ii). To see this, we observe that
		\begin{equation}\label{exam1-norm3}
			\hat N_{\mathcal C_1}((x,y,z),\Omega_1)=\hat N((x,y,z),\Omega_1\cap \mathcal C_1) = \left\{(u,0,v)\mid 0\le u=-v\right\},
			\end{equation}
		and
		\begin{equation*}
			\hat N((x,y,z),\Omega_2\cap\mathcal C_2)	=
			\begin{cases}
				\R_-^2\times\{0\}&\text{ \rm if } x=y=0;\\
				\R_-\times\{0\}^2&\text{ \rm if } x=0,y>0;\\
				\{0\}\times\R_-\times\{0\}&\text{ \rm if } x>0,y=0;\\
				\{(0,0,0)\}&\text{ \rm if } x,y>0.
			\end{cases}
		\end{equation*}
		\begin{equation} \label{exam1-norm4}
			\hat N_{\mathcal C_2}((x,y,z),\Omega_2)	=
			\begin{cases}
				\{0\}\times\R_-\times\{0\}&\text{ \rm if }x\ge 0, y=0;\\
				\{(0,0,0)\}&\text{ \rm if } x\ge 0, y>0
			\end{cases}
		\end{equation} for any $(x,y,z)$ is sufficiently close $(0,0,0).$ From \eqref{exam1-norm3} and \eqref{exam1-norm4}, we obtain that $\{\Omega_1,\Omega_2\}$ satisfy LQC wrt $\{\mathcal C_1,\mathcal C_2\}$ at $(0,0,0).$ We will show that $\{\Omega_1,\Omega_2\}$ are not  normal-densed in $\{\mathcal C_1,\mathcal C_2\}$ at $(0,0,0).$

	Taking $(x_k,y_k,z_k)=(0,0,0)$ and $(x_{ik},y_{ik},z_{ik})=(0,0,0)$ for $i=1,2$. Then
		$$\hat N_{\mathcal C_1}((x_{1k},y_{1k},z_{1k}),\Omega_1)=\hat N((x_{1k},y_{1k},z_{1k}),\Omega_1\cap \mathcal C_1) = \left\{(u,0,v)\mid 0\le u=-v\right\},$$
		and
		$$\hat N((x_{2k},y_{2k},z_{2k}),\Omega_2\cap \mathcal C_2)=\R_-^2\times\{0\},\;
		\hat N_{\mathcal C_2}((x_{2k},y_{2k},z_{2k}),\Omega_2) =\{0\}\times\R_-\times\{0\}.$$
		Picking $x_{1k}^*=2,y_{1k}^*=0,z_{1k}^*=-2$ and $x_{2k}^*=-1,y_{2k}^*=0,z_{2k}^*=0.$ Then $(x_{ik}^*,y_{ik}^*,z_{ik}^*)\in \hat N((x_{ik},y_{ik},z_{ik}),\Omega_i\cap\mathcal C_i)$ for $i=1,2$ and $(x_{1k}^*,y_{1k}^*,z_{1k}^*)\to (2,0,-2)$, $(x_{2k}^*,y_{2k}^*,z_{2k}^*)\to (-1,0,0).$ It implies that $$(x_1^*,y_1^*,z_1^*)+(x_2^*,y_2^*,z_2^*)=(1,0,-2).$$
		For any $(\tilde x_{ik},\tilde y_{ik},\tilde z_{ik})\xrightarrow{\Omega_i\cap\mathcal C_i}(0,0,0)$ and $(\tilde x^*_{ik},\tilde y^*_{ik},\tilde z^*_{ik})\in \hat N_{\mathcal C_i}((\tilde x_{ik},\tilde y_{ik},\tilde z_{ik}),\Omega_i)$ for $i=1,2$, we have $\tilde x_{2k}^*= 0, z_{2k}^*=0$ and thus $\tilde x_2^*=0, z_2^*=0.$ It follows that $\tilde x_{1k}^*\to \tilde x_1^*=1$ and so $\tilde z_{1k}^*\to \tilde z_1^*=-1.$ Therefore,
		$$(\tilde x_1^*,\tilde y_1^*,\tilde z_1^*)+(\tilde x_2^*,\tilde y_2^*,\tilde z_2^*)\ne (1,0,-2)$$ and so we can not find sequences $(\tilde x_{ik},\tilde y_{ik},\tilde z_{ik})\xrightarrow{\Omega_i\cap\mathcal C_i}(0,0,0)$ and $(\tilde x^*_{ik},\tilde y^*_{ik},\tilde z^*_{ik})$ satisfying \eqref{exam1eq1} and \eqref{exam1eq2}.

\end{example}

\begin{theorem}\label{thm-intersec-rule}
Let $\Omega_i,\mathcal C_i$ for $i=1,2$ be subsets of $\R^n$. Denote $\Omega:=\Omega_1\cap\Omega_2$ and $\mathcal C:=\mathcal C_1\cap\mathcal C_2$. Let $\{\Omega_1,\Omega_2\}$ satisfy LQC wrt $\{\mathcal C_1,\mathcal C_2\}$ at $\bar x\in\Omega\cap\mathcal C,$ and let $\{\Omega_1,\Omega_2\}$ be normal-densed in $\{\mathcal C_1,\mathcal C_2\}$ at $\bar x.$ Then the following inclusion holds.
\begin{equation}\label{thm-inter-rule-eq0}
N_{\mathcal C}(\bar x,\Omega)\subset N_{\mathcal C_1}(\bar x,\Omega_1)+N_{\mathcal C_2}(\bar x,\Omega_2).\end{equation}

Consequently, if $\mathcal C_2$ is a neighborhood of $\bar x$, then
\begin{equation}\label{thm-inter-rule-eq00}
N_{\mathcal C}(\bar x,\Omega)=N_{\mathcal C_1}(\bar x,\Omega)\subset N_{\mathcal C_1}(\bar x,\Omega_1)+N(\bar x,\Omega_2).
\end{equation}
\end{theorem}

\begin{proof} If $\bar x\in \ir\mathcal C$, then $\mathcal R(x,\mathcal C)=\R^n$ for any $x$ close $\bar x$ enough. Thus,
$$
N_{\mathcal C}(\bar x,\Omega)=N(\bar x,\Omega_1\cap\Omega_2), N_{\mathcal C_1}(\bar x,\Omega_1)=N(\bar x,\Omega_1) \text{ \rm and } N_{\mathcal C_2}(\bar x,\Omega_2)=N(\bar x,\Omega_2).
$$
By \cite[Theorem~3.4]{Mor06}, we obtain   relation~\eqref{thm-inter-rule-eq0}.

If $\bar x\in\bd \mathcal C$, then by taking $x^*\in N_{\mathcal C}(\bar x,\Omega_1\cap\Omega_2),$ one sees that there exists by the definition sequences $x_{k}\xrightarrow{\Omega\cap\mathcal C} \bar x$ and $x_k^*\to x^*$ such that $x_k^*\in \hat N(x_k,\Omega\cap\mathcal C)\cap\mathcal R(x_k,\mathcal C).$

We consider the following two case.

Case 1. $x_k\in \bd \mathcal C.$ For any sequence $\epsilon_k\to 0^+,$ by using   fuzzy intersection rule \cite[Lemma 3.1]{Mor06}, we find sequences $\lambda_k\in \R_+,$ $x_{ik}\in \Omega_i\cap\mathcal C$ and $x_{ik}^*\in \hat N(x_{ik},\Omega_i\cap\mathcal C)$ satisfying $\Vert x_{ik}-x_k \Vert\le \epsilon_k$ and
\begin{equation}\label{thm-inter-rule-eq1ii}
\Vert x_{1k}^*+x_{2k}^*-\lambda_kx_k^*\Vert\le 2\epsilon_k,\quad 1-\epsilon_k\le \max\{\lambda_k,\Vert x_{1k}^*\Vert\}\le 1+\epsilon_k,
\end{equation}
for $i=1,2.$ Since sequences $\lambda_k, x_{1k}^*$ are bounded, $x_{2k}^*$ is bounded due to \eqref{thm-inter-rule-eq1ii} and then without loss of generality, we assume that $\lambda_k\to \lambda, x_{1k}^*\to x_1^*$ and $x_{2k}^*\to x_2^*.$ Passing to the limit as $k\to \infty$ of the first inequality in \eqref{thm-inter-rule-eq1ii}, we obtain  $x_1^*+x_2^*=\lambda x^*.$ Since $\{\Omega_1,\Omega_2\}$ are normal-densed in $\{\mathcal C_1,\mathcal C_2\}$ at $\bar x,$ there exist $\tilde x_{ik}\xrightarrow{\Omega_i\cap\mathcal C_i} \bar x,$ $\tilde x_{ik}^*\in \hat N_{\mathcal C_i}(\tilde x_{ik},\Omega_i)$ for $i=1,2$ such that $\tilde x_{ik}^*\to \tilde x_i^*$ with some $\tilde x_i^*$ for $i=1,2$ and \begin{equation}\label{intersection-thm-eq001}
\lambda x^* = \tilde x_{1}^*+\tilde x_{2}^*,
\end{equation}
and that
\begin{equation} \label{intersection-thm-eq002}
\max\{\Vert\tilde x_{1}^*\Vert,\Vert\tilde x_{2}^*\Vert\}>0 \text{ \rm whenever } \max\{\Vert x_{1}^*\Vert,\Vert x_{2}^*\Vert\}>0,
\end{equation}
which  implies that $\lambda> 0$ because if in the opposite that $\lambda=0$ then $\lambda_k\to 0$ as $k\to \infty$. Since $\{\Omega_1,\Omega_2\}$ satisfy the LQC wrt $\{\mathcal C_1,\mathcal C_2\}$ at $\bar x$, we obtain from \eqref{intersection-thm-eq001} that $\tilde x_1^*=\tilde x_2^*=0$, which implies that $x_{i}^*\to 0$ for $i=1,2$  due to \eqref{intersection-thm-eq002}. This is a contradiction to the second inequality of \eqref{thm-inter-rule-eq1ii}.

Case 2. $x_k\in \ir\mathcal C.$ In this case, there exists $r_k>0$ such that $\B(x_k,r_k)\subset \mathcal C.$ Since $x_k\to \bar x\in \bd\mathcal C,$ $r_k\to 0^+.$ Setting $\tilde \epsilon_k:=\frac{r_k}{3},$ we obtain  $\tilde\epsilon_k\to 0^+$ also. By using the fuzzy intersection rule \cite[Lemma 3.1]{Mor06}, we find sequences $\tilde\lambda_k\in \R_+,$ $\tilde x_{ik}\in \Omega_i\cap\mathcal C_i$ and $\tilde x_{ik}^*\in \hat N(\tilde x_{ik},\Omega_i\cap\mathcal C_i)$ satisfying $\Vert \tilde x_{ik}-x_k \Vert\le \tilde\epsilon_k$ and
\begin{equation}\label{thm-inter-rule-eq2}
\Vert \tilde x_{1k}^*+\tilde x_{2k}^*-\tilde\lambda_kx_k^*\Vert\le 2\tilde\epsilon_k,\quad 1-\tilde\epsilon_k\le \max\{\tilde\lambda_k,\Vert \tilde x_{1k}^*\Vert\}\le 1+\tilde\epsilon_k,
\end{equation}
for $i=1,2,$ which  implies from $\Vert \tilde x_{ik}-x_k \Vert\le \tilde\epsilon_k$ that $\tilde x_{ik}\in \ir \mathcal C_i$ for $i=1,2$ and then $\mathcal R(\tilde x_{ik},\mathcal C_i)=\R^n.$ Therefore, $\hat N(x_k,\Omega\cap\mathcal C)=\hat N_{\mathcal C}(x_k,\Omega)$ and $\hat N(\tilde x_{ik},\Omega_i\cap\mathcal C_i)=\hat N_{\mathcal C_i}(\tilde x_{ik},\Omega_i)$ for $i=1,2.$

Thus, in the both cases, there exist sequences $x_{ik}\xrightarrow{\Omega_i\cap\mathcal C_i}\bar x$ and $\tilde x_{ik}^*\in \hat N_{\mathcal C_i}(x_{ik},\Omega_i)$ for $i=1,2$ satisfying $$\tilde x_{1k}^*+\tilde x_{2k}^*-\lambda_kx_k^*\to 0.$$ This follows that
$$
\lambda x^*\in N_{\mathcal C_1}(\bar x,\Omega_1)+N_{\mathcal C_2}(\bar x,\Omega_2),
$$
which implies  relation~\eqref{thm-inter-rule-eq0}.

It remains to consider the cases that  $\mathcal C_2$ is a neighborhood of $\bar x.$ In this case, we observe that $\hat N(x_{k},\Omega\cap\mathcal C)=\hat N(x_k,\Omega\cap\mathcal C_1)$ and $\mathcal R(x_k,\mathcal C)=\mathcal R(x_k,\mathcal C_1)$ and then $N_{\mathcal C}(\bar x,\Omega)=N_{\mathcal C_1}(\bar x,\Omega).$ Moreover, we also have $N_{\mathcal C_2}(\bar x,\Omega_2)=N(\bar x,\Omega_2).$ Thus we have the relations in  \eqref{thm-inter-rule-eq00}.
\end{proof}

In Theorem~\ref{thm-intersec-rule}, the assumption that $\{\Omega_1,\Omega_2\}$ are normal-densed in $\{\mathcal C_1,\mathcal C_2\}$ at $\bar x$  is {\it essential}. Indeed, we consider the following example.

\begin{example}\label{exam2}
Consider $\Omega_1,\Omega_2$ and $\mathcal C, \mathcal C_1,\mathcal C_2$ as in Example~\ref{exam1}. It is known from Example~\ref{exam1} that $\{\Omega_1,\Omega_2\}$ satisfy the LQC wrt $\{\mathcal C,\mathcal C\}$ at $\bar z=(0,0,0),$ and $\{\Omega_1,\Omega_2\}$ are normal-densed $\{\mathcal C,\mathcal C\}$ at $\bar z.$ Thus we obtain
\begin{equation}\label{exam2-eq0}
N_{\mathcal C}(\bar z,\Omega_1\cap\Omega_2)\subset N_{\mathcal C}(\bar z,\Omega_1)+N_{\mathcal C}(\bar z,\Omega_2).
\end{equation}
By directly computation (see Example~\ref{exam1}), we have
$$
N_{\mathcal C}(\bar z,\Omega_1\cap\Omega_2)=N_{\mathcal C}(\bar z,\Omega_1) \subset N_{\mathcal C}(\bar z,\Omega_1)+ N_{\mathcal C}(\bar z,\Omega_2)$$
which suits  perfectly with \eqref{exam2-eq0}.

However, by directly computation (see Example~\ref{exam1}), we also have
\begin{equation*}
	N_{\mathcal C_1\cap\mathcal C_2}((0,0,0),\Omega_1\cap\Omega_2)	=\left\{(u,0,v)\mid
	 0\le u\le -v
	\right\}
\end{equation*} while
\begin{align*}N_{\mathcal C_1}((0,0,0),\Omega_1)+ N_{\mathcal C_2}((0,0,0),\Omega_2) = \left\{(u,w,-u)\mid u\ge 0, w\le 0 \right\}.
	\end{align*} Therefore,
$$
N_{\mathcal C_1\cap\mathcal C_2}(\bar z,\Omega_1\cap\Omega_2)\nsubseteq N_{\mathcal C_1}(\bar z,\Omega_1)+ N_{\mathcal C_2}(\bar z,\Omega_1).$$ The reason is that $\{\Omega_1,\Omega_2\}$ are not normal-densed $\{\mathcal C_1,\mathcal C_2\}$ at $\bar z$ (see Example~\ref{exam1}~(ii)).

\end{example}
%
%
%
\begin{theorem}
Let $\bar x\in F^{-1}(\Theta)\cap\mathcal C$, where $F:\R^n\rightrightarrows \R^m$ is a closed-graph mapping. Assume that $\Theta$ is a closed set and $\mathcal C$ is closed and convex. Let the multifunction $x\mapsto F(x)\cap\Omega$ be inner semicompact at $\bar x$ with respect to $\mathcal C$ and the following qualification conditions be fulfilled, for every $\bar y\in F(\bar x)\cap\Theta,$

{\rm (i)} $N(\bar y,F(\Theta))\cap{\rm ker}\,D^*F^{\mathcal C}(\bar x,\bar y)=\{0\};$

{\rm (ii)} $\{\gph F,\R^n\times\Theta\}$ are normal-densed in $\{\mathcal C\times\R^m,\R^{n+m}\}$ at $(\bar x,\bar y).$

Then we have the following inclusion:
\begin{equation}\label{composite-rule-eq}
N_{\mathcal C}(\bar x,F^{-1}(\Theta))\subset \bigcup\left\{D^*_{\mathcal C}F(\bar x,\bar y)(y^*)\mid y^*\in N(\bar y,\Theta), \bar y\in F(\bar x)\cap\Theta\right\}.
\end{equation}
\end{theorem}

\begin{proof}
Taking $x^*\in N_{\mathcal C}(\bar x,F^{-1}(\Theta))$, one sees that there exist by the definition sequences $x_k\xrightarrow{F^{-1}(\Theta)\cap\mathcal C}\bar x$ and $x_k^*\in \hat N_{\mathcal C}(x_k,F^{-1}(\Theta))$ such that $x_k^*\to x^*.$ By using the inner compactness at $\bar x$ with respect to $\mathcal C$ of $\tilde F(x):=F(x)\cap\Theta$, one sees that  there exists $y_k\in F(x_k)\cap\Theta$ and its subsequence $\{y_{k_s}\}\to \bar y$ with some $\bar y\in F(\bar x)\cap\Theta.$ Without loss of generality, we can assume that   $(x_k,y_k)$ coincides with its subsequence $(x_{k_s},y_{k_s}).$  Put $\Omega_1:=\gph F$ and $\Omega_2:=\R^n\times\Theta.$ Then $\Omega_1\cap\mathcal C\times\R^m,\Omega_2$ are closed sets and $(x_k,y_k)\in \Omega_1\cap\Omega_2\cap\mathcal C\times\R^m$ for all $k.$ We next  show $(x_k^*,0)\in \hat N_{\mathcal C\times\R}((x_k,y_k),\Omega_1\cap\Omega_2).$ Indeed, for any $(x,y)\in \Omega_1\cap\Omega_2,$ we have $y\in F(x)\cap\Theta$ and $x\in F^{-1}(\Theta)\cap\mathcal C$, which implies that
$$
\limsup_{(x,y)\xrightarrow{\Omega_1\cap\Omega_2}(\bar x,\bar y)}\dfrac{\langle (x_k^*,0), (x,y)-(x_k,y_k)\rangle}{\Vert (x,y)-(x_k,y_k)\Vert} \le  \limsup_{x\xrightarrow{F^{-1}(\Theta)\cap\mathcal C}x_k}\dfrac{\langle x_k^*, x-x_k\rangle}{\Vert x-x_k\Vert}\le 0
$$
due to $x_k^*\in \hat N_{\mathcal C}(x_k,F^{-1}(\Theta))$. This follows that $(x_k^*,0)\in \hat N((x_k,y_k),\Omega_1\cap\Omega_2)$. Since $x^*\in \mathcal R(x_k, \mathcal C)$, $(x_k^*,0)\in \mathcal R((x_k,y_k),\mathcal C\times\R^m).$ Thus we derive
$$
(x^*,0)\in N_{\mathcal C\times\R^m}((\bar x,\bar y),\Omega_1\cap\Omega_2),
$$
which implies from the assumption~(ii) that $\{\Omega_1,\Omega_2\}$ are normal-densed in $\{\mathcal C_1,\mathcal C_2\}$ at $(\bar x,\bar y).$
We next prove the following claim.

{\it Claim.} $\{\Omega_1,\Omega_2\}$ satisfy the LQC wrt $\{\mathcal C_1,\mathcal C_2\}$ at $(\bar x,\bar y)$, where $\mathcal C_1:=\mathcal C\times\R^m$ and $\mathcal C_2:=\R^n\times \R^m.$ Indeed, take arbitrarily sequences $(x_{ik},y_{ik})\xrightarrow{\Omega_i\cap\mathcal C_i}(\bar x,\bar y)$ and $(x_{ik}^*,y_{ik}^*)\in \hat N_{\mathcal C_i}((x_{ik},y_{ik}),\Omega_i)$ with $(x_{ik}^*,y_{ik}^*)\to (x_i^*,y_i^*)$ for $i=1,2$ satisfying $(x_{1k}^*+x_{2k}^*, y_{1k}^*+y_{2k}^*)\to (0,0)$ as $k\to \infty.$ Then
$$
x_1^*\in D_{\mathcal C}^*F(\bar x,\bar y)(-y_1^*) \text{ \rm and } x_{2k}^*=0, \; y_{2k}^*\in \hat N(y_{2k},\Theta)\; \text{ \rm for all } k\in \N.
$$
Two last assertions imply that $x_2^*=0$ and $y_2^*\in N(\bar y,\Theta)$ and then $x_1^*=0, y_1^*=-y_2^*.$ This follows that $0\in D_{\mathcal C}^*F(\bar x,\bar y)(y_2^*)$, which gives us that $y_2^*\in {\rm ker}\,D_{\mathcal C}^*F(\bar x,\bar y)\cap N(\bar y,\Theta).$ By using the assumption~(i), we obtain $y_2^*=0$ and hence we derive that $\{\Omega_1,\Omega_2\}$ satisfies the LQC  wrt $\{\mathcal C_1,\mathcal C_2\}$ at $(\bar x,\bar y)$.

Using \eqref{thm-inter-rule-eq00} in Theorem~\ref{thm-intersec-rule}, we obtain
$$
(x^*,0) \in N_{\mathcal C\times \R^m}((\bar x,\bar y),\Omega_1\cap\Omega_2)\subset N_{\mathcal C\times\R^m}((\bar x,\bar y),\Omega_1)+N((\bar x,\bar y),\Omega_2),
$$
so we find $(x_1^*,y_1^*)\in N_{\mathcal C\times\R^m}((\bar x,\bar y),\Omega_1)$ and $(x_2^*,y_2^*)\in N((\bar x,\bar y),\Omega_2)$ such that
$$
(x^*,0)=(x_1^*,y_1^*)+(x_2^*,y_2^*).
$$
Moreover, it is not difficult to see that $N((\bar x,\bar y),\Omega_2)=\{0\}\times N(\bar y,\Theta).$ Thus $x_2^*=0$ and $y_2^*\in N(\bar y,\Theta).$ This follows that $x^*=x_1^*$ and $-y_1^*\in N(\bar y,\Theta)$, which imply that $x^*\in D^*_{\mathcal C}F(\bar x,\bar y)(y^*)$ with $y^*=-y_2^*\in N(\bar y,\Theta)$. Hence \eqref{composite-rule-eq} holds.
\end{proof}

\section{Formulas for Calculating of the Limiting Coderivative with respect to a Set}

This section is to develop  formulas for sum and chain rules of the limiting coderivatives with respect to a set of set-valued mappings. We  start with sum rules. Let us consider two set-valued mappings $F_1, F_2: \mathbb R^n\rightrightarrows\R^m$. We define a set-valued mapping $S:\R^{n+m}\rightrightarrows\R^{2m}$ and sets $\Omega_i$ for $i=1,2$ as follows, (see \cite{Mor06}):
\begin{equation}\label{s-func}
S(x,y):=\{(y_1,y_2)\mid y_1\in F_1(x), y_2\in F_2(x), y_1+y_2=y\}
\end{equation}
and
\begin{equation}\label{Omega_i}
\Omega_i:=\left\{(x,y_1,y_2)\in \R^n\times\R^m\times\R^m\mid (x,y_i)\in \gph F_i \right\}.
\end{equation}

\begin{theorem}[sum rules]\label{sum-rule}
Let $F_i:\R^n\rightrightarrows\R^m$ with $(\bar x,\bar y)\in \gph(F_1+F_2)$  and $\mathcal C_i\subset \R^n$ be nonempty and closed sets and $\bar x\in\mathcal C_i$ for $i=1,2$. Assume that the following qualifications are fulfilled:

{\rm ($q_1$) } $D^*_{\mathcal C_1}F_1(\bar x,\bar y_1)(0)\cap [-D_{\mathcal C_2}^*F_2(\bar x,\bar y_2)(0)]=\{0\},$

{\rm ($q_2$)} $\{\Omega_1,\Omega_2\}$ are normal-densed in $\{\mathcal C_1\times\R^{2m},\mathcal C_2\times\R^{2m}\}$ at $(\bar x,\bar y_1,\bar y_2),$ where $\Omega_i$ for $i=1,2$ is given by \eqref{Omega_i}.

Then we have the following assertions:

{\rm (i)} If $\gph F_1$ and $\gph F_2$ are locally closed around $(\bar x, \bar y_1)$ and $(\bar x,\bar y_2)$ respectively and $S$ is inner semicontinuous at $(\bar x,\bar y,\bar y_1,\bar y_2)$ with respect to $\mathcal C\times\R^m$ with $\mathcal C:=\mathcal C_1\cap\mathcal C_2,$ then \begin{equation}\label{sum-rule-1}
D^*_{\mathcal C}(F_1+F_2)(\bar x,\bar y)(y^*)\subset D^*_{\mathcal C_1}F_1(\bar x,\bar y_1)(y^*)+D^*_{\mathcal C_2}F_2(\bar x,\bar y_2)(y^*)\; \forall y^*\in \R^m.\end{equation}

{\rm (ii)} If $F_i$ for $i=1,2$ is closed-graph with respect to $\mathcal C$ whenever $x$ is near $\bar x$ and $S$ is inner semicompact at $(\bar x,\bar y)$ with respect to $\mathcal C\times\R^m$, then, for any $y^*\in \R^m$,
\begin{equation}\label{sum-rule-2}
D^*_{\mathcal C}(F_1+F_2)(\bar x,\bar y)(y^*)\subset \bigcup_{(\bar y_1,\bar y_2)\in S(\bar x,\bar y)}\left[D^*_{\mathcal C_1}F_1(\bar x,\bar y_1)(y^*)+D^*_{\mathcal C_2}F_2(\bar x,\bar y_1)(y^*)\right].
\end{equation}
\end{theorem}
\begin{proof}
We first prove (i). It is clearly that $\Omega_1$ and $\Omega_2$ are locally closed around $(\bar x,\bar y_1,\bar y_2)$ due to the local closedness of $\gph F_1$ and $\gph F_2$ around $(\bar x,\bar y_1)$ and $(\bar x,\bar y_2)$ respectively. Taking $(y^*,x^*)\in D^*_{\mathcal C}(F_1+F_2)(\bar x,\bar y)$, we find by the definition sequences $(x_k,y_k)\xrightarrow{\tgph(F_1+F_2)\cap \mathcal C}(\bar x,\bar y)$ and $(x_k^*, y_k^*)\to (x^*,y^*)$ such that $$(x_k^*,-y_k^*)\in \hat N_{\mathcal C\times\R^m}((x_k,y_k),\gph(F_1+F_2))$$ for all $k\in \N.$ There exist, by the inner semicontinuity with respect to $\mathcal C$ at $(\bar x,\bar y,\bar y_1,\bar y_2)$ of $S$, a sequence $(y_{k1},y_{2k})\to (\bar y_1,\bar y_2)$ such that $(y_{1k},y_{2k})\in S(x_k,y_k)$, which gives us that $(x_k,y_{1k},y_{2k})\in \Omega_1\cap\Omega_2$ for all $k.$ We next prove that
\begin{equation}\label{sum-rule-eq0}
(x_k^*,-y_k^*,-y_k^*)\in \hat N_{\mathcal C\times\R^m\times\R^m}((x_k,y_{1k},y_{2k}),\Omega_1\cap\Omega_2)\; \forall k\in \N.
\end{equation}
Indeed, for any $(x,y_1,y_2)\in \Omega_1\cap\Omega_2\cap \mathcal C,$ we have
$$
\dfrac{\langle(x_k^*,-y_k^*,-y_k^*),(x,y_1,y_2)-(x_k,y_{1k},y_{2k})\rangle}{\Vert (x,y_1,y_2)-(x_k,y_{1k},y_{2k})\Vert}=\dfrac{\langle(x_k^*,-y_k^*),(x,y)-(x_k,y_{k})\rangle}{\Vert (x,y)-(x_k,y_{k})\Vert},
$$
where $y:=y_1+y_2\in (F_1+F_2)(x)$ and $y_k:=y_{1k}+y_{2k}\in (F_1+F_2)(x_k).$ It implies that
\begin{align}
&\limsup_{(x,y_1,y_2)\xrightarrow{\Omega_1\cap\Omega_2\cap\mathcal C\times\R^{2m}}(x_k,y_{1k},y_{2k}} \dfrac{\langle(x_k^*,-y_k^*,-y_k^*),(x,y_1,y_2)-(x_k,y_{1k},y_{2k})\rangle}{\Vert (x,y_1,y_2)-(x_k,y_{1k},y_{2k})\Vert}\nonumber\\
&=\limsup_{(x,y)\xrightarrow{\tgph(F_1+F_2)\cap\mathcal C\times\R^m}(x_k,y_{k})}\dfrac{\langle(x_k^*,-y_k^*),(x,y)-(x_k,y_{k})\rangle}{\Vert (x,y)-(x_k,y_{k})\Vert}\nonumber\\
&\le 0,\label{sum-rule-eq1}
\end{align}
where \eqref{sum-rule-eq1} holds due to $(x_k^*,-y_k^*)\in \hat N_{\mathcal C\times\R^m}((x_k,y_k),\gph(F_1+F_2)).$ Passing to the limits in \eqref{sum-rule-eq0}, we have
\begin{equation}\label{sum-rule-eq3}
(x^*,-y^*,-y^*)\in N_{\mathcal C\times\R^{2m}}((\bar x,\bar y_1,\bar y_2),\Omega_1\cap\Omega_2).
\end{equation}
We now show the assertion that $\{\Omega_1,\Omega_2\}$ satisfy the LQC wrt $\{\tilde{\mathcal C_1},\tilde{\mathcal C_2}\}$ at $(\bar x,\bar y_1,\bar y_2),$ where $\tilde{\mathcal C_i}:=\mathcal C_i\times\R^{2m}$ for $i=1,2.$ To see this, we take sequences $(x_{ik},u_{ik},v_{ik})\xrightarrow{\Omega_i\cap\tilde{\mathcal C_i}}(\bar x,\bar y_1,\bar y_2)$ and $(x_{ik}^*,u_{ik}^*,v_{ik}^*)\in \hat N_{\tilde{\mathcal C_i}}((x_{ik},u_{ik},v_{ik}),\Omega_i)$ with $(x_{ik}^*,u_{ik}^*,v_{ik}^*)\to (x_i^*,u_i^*,v_i^*)$ for $i=1,2$ such that
\begin{equation}\label{sum-rule-eq2}
(x_{1k}^*,u_{1k}^*,v_{1k}^*)+(x_{2k}^*,u_{2k}^*,v_{2k}^*)\to 0.
\end{equation}
It implies from \eqref{Omega_i} and the definition of $\tilde{\mathcal C_i}$ that
$$
\Omega_i\cap\tilde{\mathcal C_i}=\{(x,y_1,y_2)\in\R^{n+2m}\mid x\in\mathcal C_i, y_i\in \gph F_i\},
$$
for $i=1,2.$
Therefore, from Theorem~\ref{product-rule2}~(i), we obtain
$$
\hat N_{\tilde{\mathcal C_1}}((x_{1k},u_{1k},v_{1k}),\Omega_1)=\left\{(x_{1k}^*,u_{1k}^*,v_{1k}^*)\mid (x_{1k}^*,u_{1k}^*)\in \hat N_{\mathcal C_1\times\R^m}((x_{1k},u_{1k}),\gph F_1), v_{1k}^*=0\right\}
$$
and
$$
\hat N_{\tilde{\mathcal C_2}}((x_{2k},u_{2k},v_{2k}),\Omega_2)=\left\{(x_{2k}^*,u_{2k}^*,v_{2k}^*)\mid (x_{2k}^*,v_{2k}^*)\in \hat N_{\mathcal C_2\times\R^m}((x_{2k},v_{2k}),\gph F_2), u_{2k}^*=0\right\},
$$
which give us that
$$
v_1^*=u_2^*=0, \; (x_1^*,u_1^*)\in N_{\mathcal C_1\times\R^m}((\bar x,\bar y_1),\gph F_1)\; \text{ \rm and } (x_2^*,v_2^*)\in N_{\mathcal C_2\times\R^m}((\bar x,\bar y_2),\gph F_2).
$$
Taking \eqref{sum-rule-eq2} into account, we have $x_1^*+x_2^*=0, u_1^*=v_2^*=0$ and $x_i^*\in D_{\mathcal C_i}^*F_i(\bar x,\bar y_i)(0)$ for $i=1,2$, so $x_1^*\in D_{\mathcal C_1}^*F_1(\bar x,\bar y_1)(0)\cap\left[-D_{\mathcal C_2}^*F_2(\bar x,\bar y_2)(0)\right]=\{0\}$ due to the assumption~($q_1$). Thus, we obtain $x_1^*=x_2^*=0$ and hence the assertion is proven.

Combining this assertion, Assumption~($q_2$), and Theorem~\ref{thm-intersec-rule}, we find by \eqref{sum-rule-eq3} elements $(x_i^*,-u_i^*,-v_i^*)\in N_{\mathcal C_i\times\R^{2m}}((\bar x,\bar y_1,\bar y_2),\Omega_i)$ for $i=1,2$ such that
$$
(x_1^*,-u_1^*,-v_1^*)+(x_2^*,-u_2^*,-v_2^*)=(x^*,-y^*,-y^*).
$$
From Theorem~\ref{product-rule2}~(ii), one has
$$
N_{\mathcal C_1\times\R^{2m}}((\bar x,\bar y_1,\bar y_2),\Omega_1)=N_{\mathcal C_1\times\R^m}((\bar x,\bar y_1),\gph F_1)\times\{0\},
$$
$$
N_{\mathcal C_2\times\R^{2m}}((\bar x,\bar y_1,\bar y_2),\Omega_2)=\{(x^*,0,v^*)\mid (x^*,v^*)\in N_{\mathcal C_2\times\R^m}((\bar x,\bar y_2),\gph F_2)\}
$$
due to the definition of $\Omega_1$ and $\Omega_2$, which  imply that $v_1^*=u_2^*=0$ and then $(x_1^*+x_2^*,-u_1^*,-v_2^*)=(x^*,-y^*,-y^*).$ This gives us that $x_1^*+x_2^*=x^*$ and $u_1^*=v_2^*=y^*.$ Moreover, we also have
$$
x_1^*\in D^*_{\mathcal C_1}F_1(\bar x,\bar y_1)(u_1^*) \; \text{ \rm and } x_2^*\in D^*_{\mathcal C_2}F_2(\bar x,\bar y_2)(v_2^*),
$$
which mean that
$$
x^*\in D^*_{\mathcal C_1}F_1(\bar x,\bar y_1)(y^*)+D^*_{\mathcal C_2}F_2(\bar x,\bar y_2)(y^*).
$$
So we obtain the relation~\eqref{sum-rule-1}.

We next justify (ii). Taking $(y^*,x^*)\in D^*_{\mathcal C}(F_1+F_2)(\bar x,\bar y),$ there exist by the definition sequences $(x_k,y_k)\xrightarrow{\tgph (F_1+F_2)\cap\mathcal C\times\R^m}(\bar x,\bar y)$ and $(x_k^*,y_k^*)\in\hat N_{\mathcal C\times\R^m}((x_k,y_k),\gph(F_1+F_2))$ such that $(x_k^*,y_k^*)\to (x^*,y^*).$ From the inner semicompactness with respect to $\mathcal C\times\R^m$ at $(\bar x,\bar y)$ of $S$, we find a sequence $(y_{1k},y_{2k})\in S(x_k,y_k)$ and a subsequence $(y_{1k_n},y_{2k_n})$ of $(y_{1k},y_{2k})$ satisfying $(y_{1k_n},y_{2k_n})\to (\bar y_1,\bar y_2)$ for some $\bar y_1,\bar y_2\in \R^m.$ It implies that $y_{ik_n}\in F_i(x_{k_n})$ for $i=1,2$ and $y_{1k_n}+y_{2k_n}=y_{k_n}$, and so $(x_{k_n},y_{ik_n})\in \gph F_i$ for all $k_n\in \N.$ Since $F_i$, for $i=1,2$, is closed-graph with respect to $\mathcal C$ at $x$ for any $x$ close enough $\bar x,$ we have $\bar y_i\in F_i(\bar x)$ for $i=1,2$. Moreover, we also have $\bar y_1+\bar y_2=\bar y.$ By using the similar arguments as in (i) with $(x_{ik_n}, y_{ik_n})$ instead of $(x_{ik}, y_{ik})$, $i=1,2,$ we obtain
$$
x^*\in D^*_{\mathcal C_1}F_1(\bar x,\bar y_1)(y^*)+D^*_{\mathcal C_2}F_2(\bar x,\bar y_2)(y^*).
$$
Thus the relation~\eqref{sum-rule-2} holds.
Hence, the proof of theorem is completed.
\end{proof}

\begin{remark}\label{rem4}
From Theorem~\ref{thm1}, we can see that the qualification condition $(q_1)$ in Theorem~\ref{sum-rule} can be replaced by the following condition: $F_1$ satisfies Aubin property with respect to $\mathcal C_1$ at $(\bar x,\bar y_1)$ or $F_2$ satisfies Aubin property with respect to $\mathcal C_2$ at $(\bar x,\bar y_2).$ Note that, the qualification condition that $F_1$ (or $F_2$) satisfies the Aubin property with respect to $\mathcal C$ at $(\bar x,\bar y_1)$ (resp. $(\bar x,\bar y_2)$) is easier to achieve than that $F_1$ (or $F_2$) satisfies the Aubin property at $(\bar x,\bar y_1)$ (resp. $(\bar x,\bar y_2)$).
\end{remark}

We next state chain rules for the limiting coderivatives with respect to a set of set-valued mappings. Given $G:\R^n\rightrightarrows \R^m$, $F:\R^m\rightrightarrows\R^s$ and $\mathcal C\subset \R^n,$
we define a set-valued mapping $\mathcal S:\R^n\times\R^s\rightrightarrows\R^m$ by $\mathcal S(x,z):=G(x)\cap F^{-1}(z)$ for every $(x,z)\in \R^n\times\R^s$ and sets $\Theta_1:=\gph G\times\R^s$, $\Theta_2:=\R^n\times\gph F.$

\begin{theorem}[chain rules]\label{chain-rule}
Let $G:\R^n\rightrightarrows \R^m$, $F:\R^m\rightrightarrows\R^s$, and $\bar z\in (F\circ G)(\bar x).$ Let $\bar y\in \mathcal S(\bar x,\bar z)$. Assume that the following qualification conditions are fulfilled:

{\rm ($q_1$)} $D^*F(\bar y,\bar z)(0)\cap \left[\ker D_{{\mathcal C}}^*G(\bar y,\bar x)\right]=\{0\}$.

{\rm ($q_2$)} $\{\Theta_1,\Theta_2\}$ are normal-densed in $\{\mathcal C\times\R^{m+s},\R^{n+m+s}\}$ at $(\bar x,\bar y,\bar z).$

Then the following assertions hold:

{\rm (i)} Let $\mathcal S$ be inner semicontinuous with respect to $\mathcal C\times\R^{m+s}$ at $(\bar x,\bar z,\bar y),$ and let $\gph G$ and $\gph F$ be locally closed around $(\bar x,\bar y)$ and $(\bar y,\bar z)$ respectively. Then
\begin{equation}\label{chain-rule-eqi}
D^*_{\mathcal C}(F\circ G)(\bar x,\bar z)(z^*)\subset D^*_{\mathcal C}G(\bar x,\bar y)\circ D^*F(\bar y,\bar z)(z^*)\; \forall z^*\in \R^s.\end{equation}

{\rm (ii)} Let $\mathcal S$ be inner semicompact with respect to $\mathcal C\times\R^{m+s}$ at $(\bar x,\bar z)$ and let $G$ and $F^{-1}$ be closed-graph with respect to $\mathcal C\times\R^{m+s}$ whenever $x$ is near $\bar x$ and $z$ is near $\bar z$ respectively. Then
$$
D^*_{\mathcal C}(F\circ G)(\bar x,\bar z)(z^*)\subset \bigcup_{\bar y\in \mathcal S(\bar x,\bar z)}\left[D^*_{\mathcal C}G(\bar x,\bar y)\circ D^*F(\bar y,\bar z)(z^*)\right]\; \forall z^*\in \R^s.
$$
\end{theorem}

\begin{proof}
It is sufficient to prove (i) because the proof of (ii) is similar. To demonstrate (i), we take $x^*\in D^*_{\mathcal C}(F\circ G)(\bar x,\bar z)(z^*).$ We find sequences $(x_k,z_k)\xrightarrow{\tgph(F\circ G)\cap\mathcal C\times\R^s}(\bar x,\bar z)$, $(x_k^*,-z_k^*)\in \hat N_{\mathcal C\times\R^s}((x_k,z_k),\gph(F\circ G))$ with $(x_k^*,z_k^*)\to (x^*,z^*).$ There exists, by the inner semicontinuous with respect to $\mathcal C\times \R^{m+s}$ at $(\bar x,\bar z,\bar y)$ of $\mathcal S,$ a sequence $y_k\to \bar y$ with $y_k\in S(x_k,z_k)$ for all $k\in \N$, which follows that $y_k\in G(x_k)$ and $z_k\in F(y_k)$. We now justify that
$$
(x_k^*,0,-z_k^*)\in \hat N_{\mathcal C\times\R^{m+s}}((x_k,y_k,z_k),\Theta_1\cap\Theta_2),\;\forall k\in\N.
$$
Indeed, for any $(x,y,z)\in \Theta_1\cap\Theta_2,$ we obtain $(x,y)\in \gph G$ and $(y,z)\in \gph F$, which imply that $z\in (F\circ G)(x).$ Moreover, we also have
\begin{align*}
\dfrac{\langle (x_k^*,0,-z_k^*), (x,y,z)-(x_k,y_k,z_k)\rangle}{\Vert(x,y,z)-(x_k,y_k,z_k)\Vert}\le \dfrac{(x_k^*,-z_k^*), (x,z)-(x_k,z_k)\rangle}{\Vert(x,z)-(x_k,z_k)\Vert}.
\end{align*}
Therefore,
\begin{align}
&\limsup_{(x,y,z)\xrightarrow{\Theta_1\cap\Theta_2\cap\mathcal C\times\R^{m+s}}(x_k,y_k,z_k)}\dfrac{\langle (x_k^*,0,-z_k^*), (x,y,z)-(x_k,y_k,z_k)\rangle}{\Vert(x,y,z)-(x_k,y_k,z_k)\Vert}\nonumber\\
&\le \limsup_{(x,z)\xrightarrow{\tgph(F\circ G)\cap\mathcal C\times\R^s}(x_k,z_k)}\dfrac{\langle(x_k^*,-z_k^*), (x,z)-(x_k,z_k)\rangle}{\Vert(x,z)-(x_k,z_k)\Vert}\nonumber\\
&      \le 0, \label{chain-rule-eq1}
\end{align}
where \eqref{chain-rule-eq1} holds due to $(x_k^*,-z_k^*)\in \hat N_{\mathcal C\times\R^s}((x_k,z_k),\gph (F\circ G)).$ This derives that $(x_k^*,0,-z_k^*)\in \hat N_{\mathcal C\times\R^{m+s}}((x_k,y_k,z_k),\Theta_1\cap\Theta_2).$ At the same time, it is clearly that
$$
(x_k^*,0,-z_k^*)\in \mathcal R((x_k,y_k,z_k), \mathcal C\times\R^{m+s}) \text{ \rm whenever } (x_k^*,-z_k^*)\in\mathcal R((x_k, z_k), \mathcal C\times \R^s).
$$
Thus we obtain
$$
(x_k^*,0,-z_k^*)\in \hat N_{\mathcal C\times\R^{m+s}}((x_k,y_k,z_k),\Theta_1\cap\Theta_2),
$$
which gives us that
\begin{equation}\label{chain-rule-eq3}(x^*,0,-z^*)\in N_{\mathcal C\times\R^{m+s}}((\bar x,\bar y,\bar z),\Theta_1\cap\Theta_2).
\end{equation}
For simply, we put $\tilde{\mathcal C_1}:=\mathcal C\times\R^{m+s}$ and $\tilde{\mathcal C_2}:=\R^{n+m+s}.$ We next prove that $\{\Theta_1,\Theta_2\}$ satisfy the LQC wrt $\{\tilde{\mathcal C_1},\tilde{\mathcal C_2}\}$ at $(\bar x,\bar y,\bar z).$ To do this, we take sequences $(x_{ik},y_{ik},z_{ik})\xrightarrow{\Theta_i\cap\tilde{\mathcal C_i}}(\bar x,\bar y,\bar z)$ and $(x_{ik}^*,y_{ik}^*,z_{ik}^*)\in \hat N_{\tilde{\mathcal C_i}}((x_{ik},y_{ik}, z_{ik}),\Theta_i)$ with $(x_{ik}^*,y_{ik}^*,z_{ik}^*)\to (x_{i}^*,y_{i}^*,z_{i}^*)$ for $i=1,2$ such that \begin{equation}\label{chain-rule-eq2}
(x_{1k}^*,y_{1k}^*,z_{1k}^*)+(x_{2k}^*,y_{2k}^*,z_{2k}^*)\to 0.
\end{equation}
From the definition of $\Theta_i$ and $\mathcal C_i$ for $i=1,2$, by Theorem~\ref{product-rule2}~(i), we have
$$
\hat N_{\tilde{\mathcal C_1}}((x_{1k},y_{1k},z_{1k}),\Theta_1)=\hat N_{\mathcal C\times\R^m}((x_{1k},y_{1k}),\gph G)\times\{0\}
$$
and
$$
\hat N_{\tilde{\mathcal C_2}}((x_{2k},y_{2k},z_{2k}),\Theta_2)=\hat N((x_{2k},y_{2k},z_{2k}),\Theta_2)=\{0\}\times\hat N((x_{2k},y_{2k}),\gph F).
$$
Taking \eqref{chain-rule-eq2} into account, we have $x_{1k}^*\to 0$, $z_{2k}^*\to 0$ and $y_{1k}^*+y_{2k}^*\to 0$, which imply that $x_1^*=x_2^*=0,$ $z_1^*=z_2^*=0$, $y_1^*=-y_2^*$, and $(x_1^*,y_1^*)\in N_{\mathcal C\times\R^m}((\bar x,\bar y),\gph G)$, $(y_2^*,z_2^*)\in N((\bar y,\bar z),\gph F).$ These follow that $y_2^*\in D^*F(\bar y,\bar z)(0)$ and $0\in D_{\mathcal C}^*G(\bar x,\bar y)(-y_1^*).$ The last relation is equivalent to $-y_1^*\in \ker D_{\mathcal C}^*G(\bar y,\bar x)$, so we obtain
$$
y_2^*\in D^*F(\bar y,\bar z)(0)\cap\left[\ker D_{\mathcal C}^*G(\bar y,\bar x)\right]
$$
and hence $y_2^*=y_1^*=0$ due to the assumption~($q_1$) and the fact that $y_1^*=-y_2^*$. This follows that $(x_i^*,y_i^*,z_i^*) =0$ for $i=1,2$, which means that $\{\Theta_1,\Theta_2\}$ satisfy the LQC wrt $\{\tilde{\mathcal C_1},\tilde{\mathcal C_2}\}$ at $(\bar x,\bar y,\bar z).$ Combining this and the assumption~($q_2$) and Theorem~\ref{thm-intersec-rule}, for $i=1,2,$ we find  $(x_i^*,y_i^*,-z_i^*)\in N_{\tilde{\mathcal C_i}}((\bar x,\bar y,\bar z),\Theta_i)$ such that
$$
(x_1^*,-y_1^*,-z_1^*)+(x_2^*,-y_2^*,-z_2^*)=(x^*,0,-z^*).
$$
By using Theorem~\ref{thm51}, we obtain $(x_1^*,-y_1^*)\in N_{\mathcal C\times\R^m}((\bar x,\bar y),\gph G)$, $z_1^*=0,$ $x_2^*=0$, $(y_2^*,z_2^*)\in N((\bar y,\bar z),\gph F)$, so $x_1^*=x^*\in D^*_{\mathcal C}G(\bar x,\bar y)(y_1^*),$ $y_1^*=-y_2^*\in D^*F(\bar y,\bar z)(z_2^*)$ and $z_2^*=z^*.$ Therefore,
$$
x^*\in D^*_{\mathcal C}G(\bar x,\bar y)\circ D^*F(\bar y,\bar z)(z^*),
$$
which gives us the relation \eqref{chain-rule-eqi}.
\end{proof}

\begin{remark}
We can see that the qualification condition $(q_1)$ in Theorem~\ref{chain-rule} can be replaced by the following condition that $F$ satisfies Aubin property at $(\bar y,\bar z)$.
\end{remark}
\section{Applications to Optimality Conditions}\label{sec-app}
We establish now necessary optimality conditions due to generlaized differentials with respect to a set for the \ref{mpecs} problem as follows.

\begin{theorem}\label{nec-con} Consider the \ref{mpecs} problem. Let $\bar x$ be a local minimizer to \eqref{mpecs}, and let $f_{\mathcal C_1}\in \mathcal F(\bar x)$. Assume that the following qualification conditions are fulfilled:

{\rm($q_1$)} $\partial_{\mathcal C_1}^{\infty}f(\bar x)\cap [-D_{\mathcal C_2}^*G(\bar x,0)(0)]=\{0\}$.

{\rm($q_2$)} $\{\Omega_1,\Omega_2\}$ are normal-densed in $\{\mathcal C_1\times\R^{m+1},\mathcal C_2\times\R^{m+1}\}$ at $(\bar x,0,f(\bar x)),$ where $\Omega_i$ for $i=1,2$ is, respectively, defined by
\begin{equation}\label{omegai}
\Omega_1:=\{(x,y,z)\mid (x,z)\in \epi f, y\in \R^m\} \; \text{ \rm and } \Omega_2:=\gph G\times\R.\end{equation}

Then
\begin{equation}\label{thm14-eq0}
0\in \partial f_{\mathcal C_1}(\bar x)+D^*_{\mathcal C_2}G(\bar x,0)(0).
\end{equation}
\end{theorem}

\begin{proof}
We first define the function $\tilde f:\R^n\times\R^m\to\bar\R$ by $\tilde f(x,y):=f(x)$ for all $(x,y)\in \R^n\times\R^m$. We consider the following optimization problem:
\begin{align}\label{optim-prob1}
\min \quad & \tilde f(x,y)\\\nonumber
\text{such that } & y\in G(x), x\in\mathcal C_1\cap\mathcal C_2.
\end{align}
It is easy to see that $\bar x$ is a local solution to \eqref{mpecs} if and only if $(\bar x,0)$ is a local solution to \eqref{optim-prob1}. Putting $\Omega:= \gph G$ and $\mathcal C:=\mathcal C_1\cap\mathcal C_2,$ the problem \eqref{optim-prob1} can be expressed as follows:
\begin{align}\label{optim-prob2}
\min \quad & \tilde f(x,y)+\delta_{\Omega}(x,y)\\\nonumber
\text{such that } & (x,y)\in\mathcal C\times \R^m.
\end{align}
By using Theorem~\ref{tq1-thm2}, we derive
\begin{align}
(0,0)\in \partial_{\mathcal C\times\R^m}(\tilde f+\delta_{\Omega})(\bar x,0)&=D_{\mathcal C\times\R^m}^*\mathcal E^{\tilde f+\delta_{\Omega}}((\bar x,0),f(\bar x))(1)\nonumber\\
&=D_{\mathcal C\times\R^m}^*\left(\mathcal E^{\tilde f}+\mathcal E^{\delta_{\Omega}}\right)((\bar x,0),f(\bar x))(1).\label{nec-con-eq3}
\end{align}
We next prove the following two claims:

{\it Claim~1}. $D_{\mathcal C_1\times\R^m}^*\mathcal E^{\tilde f}(\bar x,0,f(\bar x))(0)\cap [-D_{\mathcal C_2\times\R^m}^*\mathcal E^{\delta_{\Omega}}(\bar x,0,0)(0)]=\{(0,0)\}.$ Indeed, let $(x^*,y^*)\in D_{\mathcal C_1\times\R^m}^*\mathcal E^{\tilde f}(\bar x,0,f(\bar x))(0)\cap [-D_{\mathcal C_2\times\R^m}^*\mathcal E^{\delta_{\Omega}}(\bar x,0,0)(0)].$ Then
\begin{equation}\label{nec-con-eq1}
(x^*,y^*,0)\in N_{\mathcal C_1\times\R^{m+1}}\left((\bar x, 0,f(\bar x)),\gph\mathcal E^{\tilde f}\right)\cap [-N_{\mathcal C_2\times\R^{m+1}}\left((\bar x,0,0),\gph\mathcal E^{\delta_{\Omega}}\right)].
\end{equation}
It observes that $\gph \mathcal E^{\tilde f}=\Omega_1.$ Set $\Omega_2':=\gph \mathcal E^{\delta_{\Omega}}=\gph G\times\R_+.$ Then we have
$$
N_{\mathcal C_1\times\R^{m+1}}\left((\bar x,0,f(\bar x)),\Omega_1\right) =\left\{(x^*,y^*,z^*)\mid y^*=0, (x^*,z^*)\in N_{\mathcal C_1\times\R^{m}}((\bar x,f(\bar x)),\epi f)\right\}
$$
and
$$
N_{\mathcal C_2\times\R^{m+1}}\left((\bar x, 0,0),\Omega_2'\right) = N_{\mathcal C_2\times\R^{m}}((\bar x,0),\gph G)\times \R_-.
$$
Taking \eqref{nec-con-eq1} into account, we achieve
$$
(x^*,0)\in N_{\mathcal C_1\times\R^{m}}((\bar x,f(\bar x)),\epi f) \text{ \rm and } (-x^*,0)\in N_{\mathcal C_2\times\R^{m}}((\bar x,0),\gph G),
$$
which imply that
$$
x^*\in \partial_{\mathcal C_1}^{\infty}f(\bar x)\cap [-D_{\mathcal C_2}^*G(\bar x,0)(0)].
$$
This gives us that $x^*=0$ due to the assumption~($q_1$). Thus we obtain Claim~1.

{\it Claim 2}. $\{\tilde\Omega_1,\tilde\Omega_2\}$ are normal-densed in $\{\tilde{\mathcal C_1}\times\R^2,\tilde{\mathcal C_2}\times\R^2\}$ at $(\bar x, 0,f(\bar x),0),$ where $\tilde{\mathcal C_i}:=\mathcal C_i\times\R^m$ for $i=1,2$, and
$$\tilde\Omega_1:=\left\{(x,y,z_1,z_2)\mid (x,y,z_1)\in \gph \mathcal E^{\tilde f},  z_2\in \R\right\}=\Omega_1\times\R,$$
$$\tilde\Omega_2:=\left\{(x,y,z_1,z_2)\mid (x,y,z_2)\in \gph \mathcal E^{\delta_{\Omega}},  z_1\in \R\right\} =\Omega_2\times\R_+.$$
It is not difficult to see that
$$\hat N((x,y,z_1,z_2),\tilde\Omega_1\cap\tilde{\mathcal C_1}\times\R^2) = \hat N((x,y,z_1),\Omega_1\cap\mathcal C_1\times\R^{m+1})\times\{0\},$$
$$\hat N_{\tilde{\mathcal C_1}\times\R^2}((x,y,z_1,z_2),\tilde\Omega_1) = \hat N_{\mathcal C_1\times\R^{m+1}}((x,y,z_1),\Omega_1)\times\{0\},$$ and
$$\hat N((x,y,z_1,z_2),\tilde\Omega_2\cap\tilde{\mathcal C_2}\times\R^2) = \hat N((x,y,z_1),\Omega_2\cap\mathcal C_2\times\R^{m+1})\times\hat N(z_2,\R_+),$$
$$\hat N_{\tilde{\mathcal C_2}\times\R^2}((x,y,z_1,z_2),\tilde\Omega_2) = \hat N_{\mathcal C_2\times\R^{m+1}}((x,y,z_1),\Omega_2)\times \hat N(z_2,\R_+).$$ Moreover, $\mathcal R((x,y,z_1,z_2),\tilde{\mathcal C}\times R^2)=\mathcal R((x,y,z_1),\mathcal C\times\R^{m+1})\times\R.$ By the assumption~($q_2$) that $\{\Omega_1,\Omega_2\}$ are normal-densed in $\{\mathcal C_1,\mathcal C_2\},$ we obtain Claim~2.

{\it Claim~3}. The mapping $S: \R^n\times\R^m\times \R$, defined by
$$
S(x,y,z)=\left\{(z_1,z_2)\mid z_1\in \mathcal E^{\tilde f}(x,y), z_2\in \mathcal E^{\delta_{\Omega}}(x,y), z_1+z_2=z\right\},
$$
is inner semicontinuous with respect to $\mathcal C\times\R^m$ at $(\bar x, 0, f(\bar x), f(\bar x), 0).$ To see this, we take any $(x_k,y_k,z_k)\xrightarrow{{\rm \tiny dom} S\cap(\mathcal C\times\R^{m+1})} (\bar x,0,f(\bar x))$. Then $z_k\in  \mathcal E^{\tilde f}(x_k,y_k)+\mathcal E^{\delta_{\Omega}}(x_k,y_k)$, which means that $z_k\ge \tilde f(x_k,y_k).$ Putting $z_{1k}=z_{k}$ and $z_{2k}=0$, we obtain $z_{1k}\in \mathcal E^{\tilde f}(x_k,y_k)$, $z_{2k}\in \mathcal E^{\delta_{\Omega}}(x_k,y_k)$, and $z_{1k}+z_{2k}=z_k$.

Moreover, we also have $z_{1k}\to f(\bar x)$ and $z_{2k}\to 0.$ By using Claims (1)-(3), and Theorem~\ref{sum-rule}, we find by \eqref{nec-con-eq3} elements $(x_1^*,y_1^*)\in D^*_{\mathcal C_1\times\R^m}\mathcal E^{\tilde f}(\bar x,0,f(\bar x))(1)$ and $(x_2^*,y_2^*)\in D^*_{\mathcal C_2\times\R^m}\mathcal E^{\delta_{\Omega}}(\bar x, 0, 0)(1)$ satisfying
$$
(x_1^*,y_1^*)+(x_2^*,y_2^*)=(0,0),
$$
which gives us the following two relations
\begin{align}
(x_1^*,y_1^*,-1)&\in N_{\mathcal C_1\times\R^{m+1}}((\bar x,0,f(\bar x)),\gph \mathcal E^{\tilde f})
\end{align}
and
\begin{align} (x_2^*,y_2^*,-1)\in N_{\mathcal C_2\times\R^m}((\bar x,0,0),\gph\mathcal E^{\delta_{\Omega}}).
\end{align}
By using Theorem~\ref{product-rule2}~(ii), we have
$$
N_{\mathcal C_1\times\R^{m+1}}((\bar x, 0, f(\bar x)),\gph \mathcal E^{\tilde f})=\{(x^*,y^*,z^*)\mid y^*=0, (x^*,z^*)\in N_{\mathcal C_1\times\R}((\bar x,f(\bar x)),\epi f)\}
$$
and
$$
N_{\mathcal C_2\times\R^{m+1}}((\bar x,0,0),\gph\mathcal E^{\delta_{\Omega}})=N_{\mathcal C_2\times\R^m}((\bar x,0),\gph G))\times\R_-.
$$
Therefore, $x_1^*\in \partial_{\mathcal C_1}f(\bar x)$ and $x_2^*\in D^*_{\mathcal C_2}G(\bar x,0)(0)$, so we derive
$$
0\in \partial_{\mathcal C_1}f(\bar x) +D^*_{\mathcal C_2}G(\bar x,0)(0).
$$
Hence, the proof is completed.
\end{proof}
\begin{remark}\label{rem5}
{\rm (i)} It is really  important to see that, in usually, the necessary optimality condition for   the \ref{mpecs} problem  is ebstablished in the form of (see e.g. \cite{BM07,G13,TC19})
\begin{equation}\label{rem5-eq1}
0\in \tilde\partial f(\bar x)+\tilde D^*G(\bar x,0)(y^*)+N(\bar x,\mathcal C),
\end{equation}
where $y^*$ is an uncertain unit vector and $\tilde\partial, \tilde D^*$ are generalized differentials. However, the inclusion \eqref{rem5-eq1} is implicit because of its dependence on $y^*$. In contrast, the condition \eqref{thm14-eq0} in Theorem~\ref{nec-con} is explicit because $y^*$ is replaced by $0$.

{\rm (ii)} In the case where $G$ satisfies the Aubin property with respect to $\mathcal C_2$ at $(\bar x, 0).$ The necessary optimality condition even depends only on the objective function $f$ but not on $G$. We have the following proposition, which is easy to obtain from  Theorems~\ref{thm1} and \ref{nec-con}.
\end{remark}
\begin{proposition}\label{nec-con-1} Consider the \ref{mpecs} problem. Let $\bar x$ be a local minimizer to \eqref{mpecs}, and let $f_{\mathcal C_1}\in \mathcal F(\bar x)$. Assume that the following qualification conditions are fulfilled:
	
	{\rm($q_1$)} The multifunction $G$ satisfies the Aubin property with respect to $\mathcal C_2$ at $(\bar x,0).$
	
	{\rm($q_2$)} $\{\Omega_1,\Omega_2\}$ are normal-densed in $\{\mathcal C_1\times\R^{m+1},\mathcal C_2\times\R^{m+1}\}$ at $(\bar x,0,f(\bar x)),$ where $\Omega_i$ is given by \eqref{omegai} for $i=1,2$.	
	
	Then we can take the set $D^*_{\mathcal C_2}G(\bar x,0)(0)$ out of \eqref{thm14-eq0}, i.e., we have
	\begin{equation}\label{pro15-eq0}
		0\in \partial f_{\mathcal C_1}(\bar x).
	\end{equation}
\end{proposition}

We now close this section by illustrated examples for Theorem~\ref{nec-con} and Proposition~\ref{nec-con-1}.

\begin{example}
Let us consider the \ref{mpecs} problem with $\mathcal C_1:=\mathcal C_2:=\mathcal C:=[0,\infty)$,
$$
f(x)=x \; \text{ \rm and }\;  G(x):=\begin{cases} \R_+ &\text{ \rm if } x\ge 0;\\
    \emptyset &\text{ \rm otherwise}.\end{cases}
$$
It is easy to see that $\bar x=0$ is the unique local solution to this problem.
Moreover, $G$ satisfies the Aubin property with respect to $\mathcal C$ at $(\bar x,0).$ By Example~\ref{exam1}~(i), the assumption ($q_2$) in Propostion~\ref{nec-con-1} holds. Thus we have $$0\in \partial_{\mathcal C}f(0).$$
By the directly computation, we really have
 $$0\in \partial_{\mathcal C}f(0)=[0,1].$$
\end{example}
The final example is to show that the assumption ($q_2$) is {\it essential} in Theorem~\ref{nec-con} and Proposition~\ref{nec-con-1}.
\begin{example}
	Let us consider the \ref{mpecs} problem with $\mathcal C_1:=\R, \mathcal C_2:=[0,\infty)$,
	$$
	f(x)=x \; \text{ \rm and }\;  G(x):=\begin{cases} \R_+ &\text{ \rm if } x\ge 0;\\
		\emptyset &\text{ \rm otherwise}.\end{cases}
	$$
	It is easy to see that $\bar x=0$ is also the unique local solution to this problem.
	Moreover, $G$ satisfies the Aubin property with respect to $\mathcal C$ at $(\bar x,0).$ By Example~\ref{exam1}~(ii), the assumption ($q_2$) in Propostion~\ref{nec-con-1} fails.
	By the directly computation, we really have
	$$0\notin \partial_{\mathcal C_1}f(0)=\{1\}.$$
\end{example}

\section*{Compliance with Ethical Standards}

{\bf Data Availability Statement}  Not applicable.

\noindent
{\bf Conflict of Interests/Competing Interests}  The authors declare that they have no conflict of interest.

\noindent
{\bf Funding} This article was supported by the National Natural Science Foundation of China
under Grant No.11401152.

\end{document}